\pdfoutput=1
\documentclass[10pt]{amsart} 
\usepackage[T1]{fontenc}
\usepackage{lmodern}
\usepackage[utf8]{inputenc}
\usepackage{textcomp}
\usepackage[english]{babel}
\usepackage{amsmath,amsfonts,amssymb,color,a4,epsfig,graphics}
\usepackage{xcolor}
\usepackage{verbatim}
\usepackage{tikz}
\usepackage{url}
\usepackage[a4paper, margin=3 cm]{geometry}
\usepackage[colorlinks=true, linkcolor=blue, citecolor=red, urlcolor=blue]{hyperref}
\usepackage{mathrsfs}

\def\build#1_#2^#3{\mathrel{\mathop{\kern 0pt#1}\limits_{#2}^{#3}}}
\usetikzlibrary{arrows}

\newcommand{\T}{{\mathbb{T}}}
\newcommand{\R}{{\mathbb{R}}}
\newcommand{\C}{{\mathbb{C}}}
\newcommand{\Z}{{\mathbb{Z}}}
\newcommand{\N}{{\mathbb{N}}}

\newcommand{\Bc}{\mathcal{B}}
\newcommand{\Cc}{\mathcal{C}}
\newcommand{\Dc}{\mathcal{D}}

\newcommand{\Hc}{\mathcal{H}}
\newcommand{\Ic}{\mathcal{I}}
\newcommand{\Jc}{\mathcal{J}}
\newcommand{\Kc}{\mathcal{K}}

\newcommand{\Sc}{\mathcal{S}}

\newcommand{\Wc}{\mathcal{W}}

\newcommand{\Oc}{\mathcal{O}}

\newcommand{\Hr}{\mathscr{H}}
\newcommand{\Ran}{\mathrm{Ran}}
\newcommand{\Thr}{\operatorname{Thr}}

\newcommand{\un}{{\rm \bf {1}}}

\numberwithin{equation}{section}

\def\id{\mathop{\rm id}\nolimits}

\def\rmi{{\rm i}}
\def \lint{[\![}
\def \rint{]\!]}

\newtheorem{theorem}{Theorem}[section]
\newtheorem{proposition}[theorem]{Proposition}
\newtheorem{lemma}[theorem]{Lemma}
\newtheorem{remark}[theorem]{Remark}
\newtheorem{definition}[theorem]{Definition}
\newtheorem{corollary}[theorem]{Corollary}

\DeclareUnicodeCharacter{2013}{--}   
\DeclareUnicodeCharacter{2014}{---}  
\DeclareUnicodeCharacter{00A0}{~}    
\DeclareUnicodeCharacter{00F6}{\"o}  


\begin{document}
\title[Interior Spectral Windows and Transport on $\mathbb{Z}^d$]%
{Interior Spectral Windows and Transport for Discrete Fractional Laplacians on $d$-Dimensional Hypercubic Lattices}
\author{Nassim Athmouni}
\address{Universit- de Gafsa, Campus Universitaire 2112, Tunisie}\email{\tt nassim.athmouni@fsgf.u-gafsa.tn}
\email{\tt athmouninassim@yahoo.fr}


\begin{abstract}
We study anisotropic fractional discrete Laplacians $\Delta_{\mathbb{Z}^d}^{\vec{\mathbf{r}}}$ with exponents $\vec{\mathbf{r}}\in\mathbb{R}^d\setminus\{0\}$ on $\ell^2(\mathbb{Z}^d)$.
We establish a Mourre estimate on compact energy intervals away from thresholds.
As consequences we derive a Limiting Absorption Principle in weighted spaces, propagation estimates (minimal velocity and local decay), and the existence and completeness of local wave operators for perturbations $H=\Delta_{\mathbb{Z}^d}^{\vec{\mathbf{r}}}+W(Q)$, where $W$ is an anisotropically decaying potential of long--range type.
In the stationary scattering framework we construct the on--shell scattering matrix $S(\lambda)$, prove the optical theorem, and, under a standard trace--class assumption on $W$, establish the Birman--Krein formula $\det S(\lambda)=\exp(-2\pi i\,\xi(\lambda))$.
\end{abstract}

\maketitle
\tableofcontents
\section{Introduction and main results}

Spectral graph theory has recently experienced a resurgence due to its wide-ranging applications in
quantum lattice models, solid-state physics, and the study of discrete structures. This is particularly evident in studies of discrete Laplacians \cite{AD,AEG,BG,S,GG,Ch,Mic,Gk} and their magnetic counterparts \cite{GT,Da1,GoMo,AEG1,ABDE}.
A central tool in analyzing their essential spectrum and scattering behavior
relies on the \emph{positive commutator method}, successfully applied to
lattices such as $\Z^d$~\cite{PR,BoSa,T1},
binary trees~\cite{S,AF,GG}, general graphs~\cite{MRT},
and even nontrivial geometries including cusps, funnels, triangular or graphene lattices~\cite{T2,GoMo,AEG,AEG1,AEGJ,AEGJ2}.

A cornerstone of modern spectral theory is the Mourre commutator method \cite{GJ,GJ1,Mo81,Ge,ABG},
which provides tools to prove the Limiting Absorption Principle (LAP),
absolute continuity of the spectrum, and \emph{propagation estimates} essential to scattering theory.
While this framework is well developed for \emph{local operators}
(such as nearest-neighbour discrete Laplacians),
its extension to \emph{nonlocal} and \emph{fractional} discrete Laplacians
introduces new analytic challenges.

Classical discrete Laplacians capture only \emph{short-range, nearest-neighbour} interactions.
Modern models of anomalous diffusion and long-range transport, however,
require operators whose influence decays \emph{algebraically} rather than being compactly supported.
This motivates the study of \emph{fractional powers} of the discrete Laplacians~\cite{JKLQ,BDL,K2017,Or,Ortigueira2014Riesz} and their magnetic perturbations~\cite{Fan2021}:
they naturally describe long-range effects on graphs~\cite{TZ},
generate nonlocal transport phenomena~\cite{ScSo2012},
and provide a bridge between discrete models and continuous nonlocal PDEs.

Two complementary viewpoints coexist in the literature.
The first one, based on the spectral theorem, exploits the fact that $\Delta_{\mathbb{Z}^d}$ is a self-adjoint and bounded operator to define $\Delta^{\vec{\mathbf{r}}}_{\mathbb{Z}^d}$ as a spectral multiplier:
\[
\Delta^{\vec{\mathbf{r}}}_{\mathbb{Z}^d} = \int_{\sigma(\Delta_{\mathbb{Z}^d})} \lambda^s \, dE_\lambda,
\]
where $\sigma(\Delta_{\mathbb{Z}^d})$ corresponds to the spectrum of $\Delta_{\mathbb{Z}^d}$ and $E_\lambda$ its spectral projection (For more details, see \cite{RS,SI2015,Dav}.).

This construction, intrinsic and coordinate-free, naturally extends to negative or even complex exponents.
It also highlights the deeply nonlocal nature of the operator:
fractional powers inherently induce long-range interactions between
distant points in configuration space~\cite{ScSo2012,K2017}.
\\The second approach is a \emph{Series or kernel representations}.
    Explicit expansions in terms of lattice shifts or finite differences
    reveal how distant sites influence each other~\cite{K2017},
    a key point for both numerics and decay analysis.
\smallskip
In the translation-invariant case of the infinite lattice $\mathbb{Z}^d$, the situation is considerably simpler. The discrete Laplacian is diagonalized by the Fourier transform, with symbol
\[
\widehat{\Delta_d}(\xi) = -\,4\sum_{j=1}^d \sin^2(\pi \xi_j), \qquad \xi\in[-\tfrac12,\tfrac12]^d,
\]
which provides a natural starting point for the spectral definition of fractional powers~\cite{RS,Dav}. However, as the dimension grows or when anisotropy and magnetic perturbations are present~\cite{Fan2021,L2000}, the increasing complexity of this symbol makes tractable power series expansions increasingly elusive.

\smallskip
A parallel line of research initiated by Ortigueira has focused on developing a
discrete-time fractional calculus. In \cite{Or}, the distinction between
Riesz potentials and fractional Laplacians was emphasized in the signal-processing
context, while his survey \cite{Ortigueira2014Riesz} traced the origins and evolutions
of discrete fractional difference operators and introduced new formalisms. More
recently, in \cite{Ortigueira_MD_DMDTFT_2024} he proposed a multidimensional
DTFT-based construction of fractional derivatives, essentially in the isotropic case
$\alpha>0$. Our contribution is of a complementary nature: instead of a
transform-based calculus, we construct anisotropic fractional Laplacians
$\Delta_{\Z^d}^{\vec r}$ with general exponents $r_j\in\R$ as self-adjoint operators
on $\ell^2(\Z^d)$, analyze their spectra and domains (including negative orders), and
develop a full operator-theoretic framework leading to Mourre estimates, limiting
absorption principles, and scattering. This technical shift---from signal-processing
motivations to spectral and dynamical analysis---is at the core of the novelty of the
present work, placing it at the crossroads between harmonic analysis, spectral theory,
and mathematical physics.
.

\smallskip

In this paper, we clarify the interplay between spectral and algebraic approaches to fractional discrete Laplacians.
We analyze the structure and limitations of power-series expansions in dimension one and higher,
and investigate their behavior under domain restrictions such as $\N$.
Our approach combines spectral theory, discrete Fourier analysis, and commutator methods,
and lays the groundwork for a rigorous scattering/propagation theory for nonlocal operators on discrete structures.

Let $\Z^d$ denote the standard \(d\)-dimensional lattice and let $\ell^2(\Z^d)$ be the Hilbert space of square-summable sequences, with norm
\[
\|f\|^2=\sum_{n\in\Z^{d}}|f(n)|^2.
\]
We let $\Cc_c(\Z^d)$ denote the dense subspace of finitely supported functions.
For each \(j\in\{1,\dots,d\}\), define the unitary shifts \(U_j,U_j^*\) by
\[
(U_j f)(n):=f(n+e_j),\qquad (U_j^* f)(n):=f(n-e_j),
\]
where \(e_j\) is the \(j\)-th canonical basis vector. The position operator \(Q_j\) acts by
\[
(Q_j f)(n):=n_j f(n).
\]
Given a bounded sequence $F:\Z^d\to\C$ and $g\in\ell^2(\Z^d)$, the multiplication operator is
$(F(Q)g)(n)=F(n)g(n)$.
For $d\ge1$ set $\langle n\rangle:=(1+|n|^2)^{1/2}$ and define the discrete Schwartz space
\[
\Sc(\Z^d):=\Bigl\{\varphi:\Z^d\to\C\ \Big|\ \forall m\in\N,\
\|\varphi\|_m:=\sup_{n\in\Z^d}\langle n\rangle^{\,m}|\varphi(n)|<\infty\Bigr\}.
\]
With the Fr\'echet topology generated by the seminorms $\|\cdot\|_m$, its continuous dual
$\Sc'(\Z^d)$ is the space of tempered distributions on $\Z^d$.
Under the canonical pairing $\langle u,\varphi\rangle:=\sum_{n\in\Z^d}u(n)\varphi(n)$ we have
\[
\Sc'(\Z^d)=\Bigl\{f:\Z^d\to\C\ \Big|\ \exists m,C>0:\ |f(n)|\le C\,\langle n\rangle^{\,m}\ \forall n\Bigr\}.
\]
We define fractional powers first on $\Sc(\Z^d)$ (equivalently on $\Cc_c(\Z^d)$) via spectral/functional calculus (or, under the Fourier transform, as multipliers with symbol $\vartheta_{\vec{\mathbf r}}(\theta)=\sum_{j=1}^d2^{r_j}(1-\cos\theta_j)^{r_j}$), and then extend them to $\Sc'(\Z^d)$ by duality; in this way $\Delta^{\vec{\mathbf r}}_{\Z^d}$ is viewed as a continuous map $\Sc'(\Z^d)\!\to\!\Sc'(\Z^d)$ with a tempered-distribution kernel; see Proposition~\ref{prop:all-r}.

In what follows we focus on the \emph{fractional discrete anisotropic Laplacian}
$\Delta^{\vec{\mathbf r}}_{\Z^d}$ acting on $\ell^2(\Z^d)$, where $\vec{\mathrm{r}}=(r_1,\dots,r_d)\neq0$:
\[
\Delta^{\vec{\mathrm{r}}}_{\Z^d}
:=\sum_{j=1}^d\mathrm{id}^{\otimes(j-1)}\otimes\Delta_{\Z}^{\,r_j}\otimes\mathrm{id}^{\otimes(d-j)}.
\]
We consider the perturbed operator
\[
H:=\Delta^{\vec{\mathrm{r}}}_{\Z^d}+W(Q),
\]
where the potential $W$ decays \emph{anisotropically} at infinity, in the sense that
\begin{enumerate}
\item[\textbf{(H0)}] $\lim_{\|n\|\to\infty}W(n)=0$;
\item[\textbf{(H1)}] $|W(n+e_j)-W(n)|\le C\,\Lambda(n)^{-\delta}\,\langle n_j\rangle^{-1}$ for all $j$,
\end{enumerate}
with $\Lambda(n):=\sum_{j=1}^d\langle n_j\rangle$.
The fractional discrete anisotropic Laplacian is essentially self-adjoint on $\Cc_c(\Z^d)$
(as a tensor sum of essentially self-adjoint one-dimensional operators); see Proposition~\ref{ESA}.
Moreover, its spectrum is
\[
\sigma\!\left(\Delta_{\Z^d}^{\vec{\mathrm{r}}}\right)
=\Bigl\{\ \sum_{j=1}^d\lambda_j\ \Big|\ \lambda_j\in\sigma\!\left(\Delta_{\Z}^{\,r_j}\right)\ \Bigr\},
\]
see Proposition~\ref{prop:spect}, and its finite threshold set is
\[
\Thr_{\mathrm{fin}}\!\bigl(H_{\Z^d}^{\vec{\mathrm{r}}}\bigr)
=\Bigl\{\ \sum_{j\in N}4^{\,r_j}+\sum_{j\in P}\epsilon_j\,4^{\,r_j}\ :\ \epsilon_j\in\{0,1\}\ \Bigr\},
\]
where $P:=\{j:\ r_j>0\}$ and $N:=\{j:\ r_j<0\}$; see Proposition~\ref{prop:thresholds-mixed}.

\begin{theorem}\label{thm:main}
Let $H=\Delta_{\Z^d}^{\vec{\mathrm r}}+W(Q)$ with real $W$ satisfying \textbf{(H0)} and \textbf{(H1)}.
Then $H\in\mathcal{C}^{1,1}_{\mathrm{loc}}(A_{\Z^d})$ on each
$\Ic\Subset\sigma(H_0)^\circ\setminus\Thr(\Delta_{\Z^d}^{\vec{\mathrm{r}}})$,
and localized Mourre theory yields:
\begin{enumerate} \renewcommand{\labelenumi}{\roman{enumi})}
\item \emph{Spectral purity.} $\sigma_{\mathrm{sc}}(H)\cap\Ic=\varnothing$ and $\sigma_{\mathrm{pp}}(H)\cap\Ic$ is finite.
\item \emph{(LAP)} The limits $(H-\lambda\mp i0)^{-1}$ exist as bounded operators
$\langle\Lambda(Q)\rangle^{-s}\ell^2(\Z^d)\to\langle\Lambda(Q)\rangle^{s}\ell^2(\Z^d)$
uniformly for $\lambda\in\Ic$ and all $s>\tfrac12$.
\item \emph{Propagation/local decay.} For all $s>\tfrac12$ and $\varphi\in C_c^\infty(\Ic)$,
\[
\int_{\R}\big\|\langle\Lambda(Q)\rangle^{-s}\,e^{-itH}\,\varphi(H)\,\langle\Lambda(Q)\rangle^{-s}f\big\|^2\,dt
\ \le\ C\,\|f\|^2.
\]
\item \emph{Asymptotic completeness.} For any $\chi\in C_c^\infty(\R)$ with $\chi\equiv1$ on $\Ic$, the local wave operators
\[
\Wc_\pm(H,H_0;\Ic):=\mathrm{s}\!-\!\lim_{t\to\pm\infty}e^{-itH}\,\chi(H_0)\,e^{itH_0}
\]
exist and satisfy $\mathrm{Ran}\,\Wc_\pm(H,H_0;\Ic)=E_\Ic(H_{\mathrm c})\,\ell^2(\Z^d)$.
In particular, scattering is asymptotically complete on $\Ic$.
\end{enumerate}
\end{theorem}

Items i)-iv) are standard consequences of localized Mourre theory on interior energies.
Once a Mourre estimate holds on $\Ic\Subset\sigma(H_0)^\circ\setminus\Thr$ and $H\in\mathcal{C}^{1,1}_{\mathrm{loc}}(A_{\Z^d})$,
\cite[Thms.~7.4.1-7.4.2]{ABG} yield the absence of singular continuous spectrum and the LAP ii)
with weights $\langle\Lambda(Q)\rangle^{\pm s}$ for all $s>\tfrac12$. In particular, on $\Ic$ the spectrum of $H$ is purely absolutely
continuous up to the finitely many eigenvalues allowed by i). Moreover, the LAP is equivalent to the local $H$-smoothness
of $\langle\Lambda(Q)\rangle^{-s}$ on $\Ic$ (Kato smoothness), which implies the local decay estimate in ~ iii); see
\cite[Ch.~7]{ABG} and \cite[Section VIII.C]{RS}. In the spectral representation this yields a Riemann-Lebesgue type decay:
for any $\varphi\in C_c^\infty(\Ic)$, any finitely supported $g$, and any $f$ in the absolutely continuous subspace of $H$,
\[
\langle g,\,e^{-itH}\,\varphi(H)f\rangle \xrightarrow[|t|\to\infty]{} 0,
\]
expressing that amplitudes leave any fixed spatial region as $|t|\to\infty$. Finally, ~iv)
(existence and completeness of the local wave operators on $\Ic$) follows from the LAP and the Mourre framework on interior energies; see again \cite[Ch.~7]{ABG}.

This work develops a localized (nonlocal) Mourre theory for fractional discrete Laplacians
$\Delta^{\vec{\mathrm{r}}}_{\Z^d}$ with anisotropic orders $\vec{\mathrm{r}}\neq0$.
Within this unified framework we prove, on compact energy windows
away from thresholds, the Limiting Absorption Principle and asymptotic completeness for fractional
difference operators. To the best of our knowledge, this is the first rigorous Mourre/LAP/scattering
treatment for fractional discrete Laplacians.\\
\noindent
The proof of Theorem~\ref{thm:main} relies on Mourre's positive commutator method \cite{Mo83,Ge2015,GGM1,GeGe,PR}, adapted to the discrete and nonlocal structure of fractional Laplacians, and rests on three explicit inputs:
(1) the existence of a \emph{self-adjoint conjugate operator} $A_{\Z^d,\vec{\mathrm r}}$ implementing discrete dilations on $\Z^d$;
(2) \emph{localized regularity away from thresholds}, namely $H_0=\Delta_{\Z^d}^{\vec{\mathrm r}}\in\mathcal{C}^2_{\mathrm{loc}}(A_{\Z^d,\vec{\mathrm r}})$ (hence $\mathcal{C}^{1,1}_{\mathrm{loc}}$) and $W(Q)\in\mathcal{C}^{1,1}(A_{\Z^d,\vec{\mathrm r}})$;
(3) a \emph{strict Mourre estimate} on each interior interval $\Ic\Subset\sigma(H_0)^\circ\setminus\Thr(\Delta_{\Z^d}^{\vec{\mathrm r}})$, stable under the perturbation $W$ with a compact remainder for $H$.
Together, (1)-(3) yield (i) spectral purity on $\Ic$ and (ii) the limiting absorption principle by \cite[Ch.~7, Thms.~7.4.1-7.4.2]{ABG} (see also \cite[App.~A]{ABG} for the Helffer-Sj{\"o}strand calculus and \cite[Thm.~VI.16]{RS} for Weyl's theorem).
In turn, (ii) implies (iii) propagation/local decay (via Kato $A$-smoothness) and (iv) asymptotic completeness (Cook-Kuroda method and abstract scattering theory; \cite[Ch.~7]{ABG}).

\medskip
We now describe the structure of the paper.
 In Section~2, we set the functional framework, recall discrete fractional Laplacians on $\Z$ and $\Z^d$,
records basic spectral/combinatorial identities (including normal-ordering expansions),
and introduces the commutator structure on $\Z^d$.
Section~3 handles anisotropically decaying perturbations $W(Q)$, proves the localized LAP,
propagation estimates, and the existence/completeness of local wave operators, yielding
Theorem~\ref{thm:main}. Section~4 presents applications: on interior energy windows, the stationary
representation of the scattering matrix (unitarity and the optical theorem), the Birman-Krein formula,
and a time-averaged ballistic transport.

\noindent\textbf{Notation.}
We denote by $\N$ the set of nonnegative integers (so $0\in\N$), and by $\lint a,b\rint:=[a,b]\cap\Z$.
We write $\un_X$ for the indicator of a set $X$. We denote by $\Kc(\Hc)$ the ideal of compact operators
on a separable Hilbert space $\Hc$. For sets $A,B$, $A\times B$ is the Cartesian product; if $A,B\subset\R$,
then $A\cdot B:=\{xy:\ (x,y)\in A\times B\}$.

\section{Functional framework for d-hypercubic Bravais lattices}
Several properties of the discrete fractional Laplacian $\Delta_{\mathbb{Z}}^{\vec{\mathbf{r}}}$ such as its diagonalization via Fourier transform, its representation as a convolution operator, and its binomial series expansion are known in the literature in various forms, particularly in the context of translation-invariant or convolution operators on abelian groups. However, these results are often stated in abstract harmonic analysis language, without explicit formulas adapted to the fractional setting or without a unified treatment of both positive and negative powers $r \in \mathbb{R}$.

In this section, we provide a self-contained and detailed derivation of these results in the discrete one-dimensional setting. Our aim is twofold. First, we adapt the classical theory to a framework that is directly usable for spectral and commutator estimates in later sections (e.g., Mourre theory, limiting absorption principle, and propagation estimates). Second, we present fully explicit expressions-such as the series expansion involving shift operators $U$, $U^*$ and binomial coefficients-that are rarely written out in the literature but play a crucial role in our analysis.

These structural properties will be essential for establishing regularity with respect to conjugate operators, studying spectral stability, and analyzing dynamics of nonlocal evolution equations on $\mathbb{Z}$ and related domains.

\subsection{Fractional powers of discrete Laplacians in the unidimensional setting}
The (forward) shift operator \(U\) on \(\ell^2(\mathbb{Z})\) is defined by
\[
(Uf)(n):=f(n-1),\qquad (U^*f)(n):=f(n+1).
\]
Then \(U\) is unitary and \(UU^*=U^*U=\mathrm{id}_{\ell^2(\mathbb{Z})}\).

We also denote by \(Q\) the position operator on \(\ell^2(\mathbb{Z})\) defined by \((Q f)(n) = n f(n)\).
For any \(s \in \mathbb{R}\), we define the weighted space \(\ell^2_s(\mathbb{Z}) := \{f \in \ell^2(\mathbb{Z}) \mid \| \langle Q \rangle^s f \| < \infty\}\), where \(\langle Q \rangle := (1 + Q^2)^{1/2}\).
Let \(\Delta_{\mathbb{Z}} := 2 \mathrm{id} - U - U^*\) denote the standard discrete Laplacian on \(\ell^2(\mathbb{Z})\).
This operator is self-adjoint and nonnegative on \(\ell^2(\mathbb{Z})\).
To extend the notion of fractional powers \(\Delta^r_{\mathbb{Z}}\) for any real \(r\), we use the discrete Fourier transform framework and the Borel functional calculus; this defines \(\Delta^r_{\mathbb{Z}}\) also for \(r<0\) (as an unbounded operator with appropriate domain).

\begin{proposition}\label{prop:all-r}
For $r\in\R$ and $u\in\mathcal S'(\Z)$,
\[
\widehat{\Delta_{\Z}^{\,r}u}(\theta)=\bigl(2(1-\cos\theta)\bigr)^r\,\widehat u(\theta),\qquad
(\Delta_{\Z}^{\,r}u)(n)=\sum_{k\in\Z}a_r(k)\,u(n-k),
\]
where
\[
a_r(k)=\frac{1}{2\pi}\int_{-\pi}^{\pi}\bigl(2(1-\cos\theta)\bigr)^r e^{\mathrm{i}k\theta}\,d\theta.
\]
The kernel $a_r$ is real-valued and even: $a_r(k)=a_r(-k)$.
\end{proposition}

\begin{proof}
Let $\mathcal F:\ell^2(\Z)\to L^2([-\pi,\pi])$ be the unitary discrete Fourier transform
\[
(\mathcal F u)(\theta)=\frac{1}{\sqrt{2\pi}}\sum_{n\in\Z}u(n)e^{-\mathrm{i}n\theta}.
\]
A direct computation gives
\[
(\mathcal F U \mathcal F^{*}f)(\theta)=e^{-\mathrm{i}\theta}f(\theta),\qquad
(\mathcal F U^{*} \mathcal F^{*}f)(\theta)=e^{\mathrm{i}\theta}f(\theta),
\]
hence
\begin{equation}\label{b-fourier}
\mathcal F\,\Delta_{\Z}\,\mathcal F^{*}
=\mathcal F\,(2I-(U+U^*))\,\mathcal F^{*}
=M_{\,2-e^{\mathrm{i}\theta}-e^{-\mathrm{i}\theta}}
=M_{\vartheta},\quad \vartheta(\theta):=2(1-\cos\theta).
\end{equation}
By the Borel functional calculus,
\(\mathcal F\,\Delta_{\Z}^{\,r}\,\mathcal F^{*}=M_{\vartheta^r}\),
whence the multiplier identity on $\ell^2(\Z)$; the extension to $\mathcal S'(\Z)$ follows by density/duality.
Define $a_r$ by $\widehat{a_r}=\vartheta^r$ (as a tempered distribution when needed). Then
\(\widehat{a_r*u}=\widehat{a_r}\,\widehat u=\vartheta^r\widehat u=\widehat{\Delta_\Z^{\,r}u}\),
and injectivity of $\mathcal F$ on $\mathcal S'(\Z)$ yields $\Delta_\Z^{\,r}u=a_r*u$.
Evenness and reality are clear from the integrand.
\end{proof}

\noindent\textbf{Regularity and summability of the kernel.}
\begin{itemize}
\item \emph{Integer $r\in\mathbb{N}^*$:} $a_r$ has finite support (local stencil).
\item \emph{Positive $r>0$ (non-integer allowed):} $a_r\in\ell^1(\mathbb{Z})$ and, as $|k|\to\infty$,
\[
a_r(k)=c_r\,|k|^{-1-2r}+\Oc \bigl(|k|^{-3-2r}\bigr),\qquad c_r\neq0.
\]
\item \emph{$r=0$:} $a_0=\delta_0$.
\item \emph{Negative $r<0$:} write $r=-s$ with $s>0$. Near $\theta=0$ the integrand behaves like $|\theta|^{2r}$.
For $-\tfrac{1}{2}<r<0$ (i.e.\ $0<s<\tfrac{1}{2}$) the integral defining $a_r(k)$ converge absolutely and
\[
a_{-s}(k)=c_{-s}\,|k|^{\,2s-1}+\Oc \bigl(|k|^{\,2s-3}\bigr)\quad(|k|\to\infty).
\]
For $r\le -\tfrac{1}{2}$ ($s\ge\tfrac12$), $a_r$ is only a tempered distribution. In all cases $r<0$, $a_r\notin\ell^1(\Z)$.
\end{itemize}

\noindent
While the diagonalization principle for $r>0$ is classical-indeed, the discrete Laplacian $\Delta_{\mathbb{Z}}$ is diagonalized by the Fourier transform with symbol $\vartheta(\theta)$, so that its fractional powers act as Fourier multipliers $\vartheta_r(\theta)$ and, by inverse transform, as convolution operators with explicit kernels (see, e.g., Reed--Simon~\cite[Chapter XIII.13]{RS}, Cycon-Froese-Kirsch-Simon~\cite[Chapter 5]{cl}, and Bucur-Valdinoci~\cite[Section 3]{BuVa}) the case $r<0$, corresponding to fractional inverse operators, is rarely discussed in detail. Here we provide a unified and explicit description of the spectrum for all $r\in\mathbb{R}$ in terms of the image of the symbol's endpoint interval $[0,4]$ under $t\mapsto t^{\,r}$. This precise characterization, together with the convolution-kernel representation inherited from $\vartheta_r$,
 is particularly useful for stability analysis, the Limiting Absorption Principle, and propagation estimates.

\begin{proposition}\label{prop:spec-all-r}
For $r\ge 0$,
\[
\sigma(\Delta_{\Z}^{\,r})=[0,\,4^{\,r}],
\]
and for $r<0$,
\[
\sigma(\Delta_{\Z}^{\,r})=[\,4^{\,r},\,\infty).
\]

\end{proposition}

\begin{proof}
Since $\mathcal F\Delta_{\Z}\mathcal F^*=M_{\vartheta}$ with continuous
$\sigma(\theta)\in[0,4]$, we have
$\sigma(\Delta_{\Z}^{\,r})=\overline{\vartheta(\theta)^r\,[ -\pi,\pi]}$,
yielding the stated intervals. If $\Delta_{\Z}^{\,r}f=\lambda u$ in $\ell^2(\Z)$, then
$\vartheta(\theta)^r\widehat f(\theta)=\lambda \widehat f(\theta)$ a.e.,
so $\widehat u$ is supported in the level set $\{\theta:\vartheta(\theta)^r=\lambda\}$,
which has Lebesgue measure $0$ (finite set) for every $\lambda$.
Hence $\widehat f=0$ in $L^2$ and $f=0$.
\end{proof}

\begin{lemma}\label{marwa}
We have $\{0,4\}\cap\sigma_p(\Delta_{\Z})=\varnothing$ in $\ell^{2}(\Z)$ (where $\sigma_p(\Delta_{\Z})$ denote the point spectrum of $\Delta_{\Z}$).

\end{lemma}

\begin{proof}
We will reason by the absurd and we assume that $0 \hbox{ or } 4\in\sigma_{p}(\Delta^r_{\Z})$.
Let $f\in\ell^2(\Z) \hbox{ such that } f\neq0$. We have
 \[\langle Uf,Uf\rangle+\langle U^*f,U^*f\rangle+\langle U^*f,Uf\rangle+\langle Uf,U^*f\rangle=4\|f\|^2.\]
Since $U$ is unitary, then
\[\langle Uf,Uf\rangle+\langle U^*f,U^*f\rangle-\langle U^*f,Uf\rangle-\langle Uf,U^*f\rangle=0.\]
Hence, $\langle(U-U^*)f,(U-U^*)f\rangle=0$. Then $(U-U^*)f=0$ for all $n\in\Z$ $f(n+1)-f(n-1)=0$.
Therefore, $f\!\restriction_{\overline{q}}$ is constant for all $\overline{q}\in\Z/2\Z$.
thus, $\|f\|^2_{\ell^2(\Z)}=\sum_{n\in\Z; \overline{n}=1}|c_1|^2+\sum_{n\in\Z;\overline{ n}=0}|c_2|^2=+\infty$.
\end{proof}

The fractional discrete Laplacian \( \Delta_{\mathbb{Z}}^r \) can be expressed as a power series in the symmetric shift operator \( \tfrac{1}{2}(U + U^*) \) using the generalized binomial expansion. The resulting expression is a combination of powers of the shift operators \( U \) and \( U^* \), weighted by combinatorial coefficients. While this structure is implicit in many works on fractional difference operators (see, e.g., Ortigueira~\cite{Or}), the fully explicit form in terms of double binomial sums is rarely written out and is detailed here for completeness.

\begin{proposition}\label{f-calculus}
Let \( r \in \mathbb{R} \). Then for every \( f \in \mathcal{C}_c(\mathbb{Z}) \), we have:
\begin{equation}\label{eq:binom-DeltaZ}
\Delta_{\mathbb{Z}}^r f(n)
= \sum_{h=0}^\infty (-1)^h\, 2^{r-h} \binom{r}{h}
\sum_{k=0}^h \binom{h}{k} U^{h - 2k} f(n),
\end{equation}
where the generalized binomial coefficients \( \binom{r}{h} \) are defined using the Gamma function:
\[
\binom{r}{h} := \frac{\Gamma(r+1)}{\Gamma(h+1)\Gamma(r - h + 1)}.
\]
\end{proposition}

\begin{proof}
We start from the functional definition of the fractional Laplacian on \( \mathbb{Z} \):
\[
\Delta^r_{\mathbb{Z}} = 2^r \left( \mathrm{id}_{\ell^2(\Z)} - \tfrac{1}{2}(U + U^*) \right)^r.
\]

We now expand the operator power \( (I - T)^r \) using the generalized binomial theorem and Lemma \ref{marwa}, with \( T := \tfrac{1}{2}(U + U^*) \). For any \( r \in \mathbb{R} \), the expansion holds as a strongly convergent series on the dense subspace \( \mathcal{C}_c(\mathbb{Z}) \):
\[
(\mathrm{id}_{\ell^2(\Z)} - T)^r = \sum_{h=0}^\infty (-1)^h \binom{r}{h} T^h.
\]

Therefore, applying this expansion to \( f \in \mathcal{C}_c(\mathbb{Z}) \), we obtain:
\[
\Delta^r_{\mathbb{Z}} f
= 2^r \sum_{h=0}^\infty (-1)^h \binom{r}{h} \left( \tfrac{1}{2}(U + U^*) \right)^h f
= \sum_{h=0}^\infty (-1)^h 2^{r - h} \binom{r}{h} (U + U^*)^h f.
\]

Next, we compute the \( h \)-th power of the symmetric shift operator:
\[
(U + U^*)^h = \sum_{k=0}^h \binom{h}{k} U^{h - 2k},
\]
where we used the binomial identity for commuting  inverse operators \( U, U^* \).

Combining the two expansions gives:
\[
\Delta^r_{\mathbb{Z}} f(n)
= \sum_{h=0}^\infty (-1)^h 2^{r - h} \binom{r}{h} \sum_{k=0}^h \binom{h}{k} U^{h - 2k} f(n).
\]

Since \( f \in \mathcal{C}_c(\mathbb{Z}) \), only finitely many terms in \( U^{h - 2k} f(n) \) are nonzero for fixed \( n \), so the double series converges absolutely for each \( n \in \mathbb{Z} \). This completes the proof.
\end{proof}
Now, we deal with the question of the commutativity of $\Delta_{\Z}^r$.
\begin{lemma}\label{sb-space}
Let $r,s \in \mathbb{R}$ and define
\[
\mathcal{H}^{s \to r} := \{ f \in \mathcal{D}(\Delta_{\mathbb{Z}}^{s}) \;:\; \Delta_{\mathbb{Z}}^{s} f \in \mathcal{D}(\Delta_{\mathbb{Z}}^{r}) \}.
\]
Then:
\begin{enumerate}
\item If $r, s \ge 0$, one has
\[
\mathcal{H}^{s \to r} \cap \mathcal{H}^{r \to s} = \mathcal{D}(\Delta_{\mathbb{Z}}^{r+s}) = \ell^2(\mathbb{Z}).
\]
\item If $r, s < 0$, one has
\[
\mathcal{H}^{s \to r} \cap \mathcal{H}^{r \to s} = \mathcal{D}(\Delta_{\mathbb{Z}}^{r+s}).
\]
\item If $s \ge 0$ and $r \le 0$, then
\[
\mathcal{H}^{s \to r} \cap \mathcal{H}^{r \to s} = \mathcal{D}(\Delta_{\mathbb{Z}}^{r}) \subsetneq \ell^2(\mathbb{Z}),
\]
and in particular:
\begin{itemize}
\item if $r = 0$, equality holds with $\ell^2(\mathbb{Z})$;
\item if $r+s < 0$, then $\mathcal{H}^{s \to r} \cap \mathcal{H}^{r \to s} \subsetneq \mathcal{D}(\Delta_{\mathbb{Z}}^{r+s})$.
\end{itemize}
\end{enumerate}
\end{lemma}

\begin{proof}
Applying \eqref{b-fourier}, the operator $\Delta_{\mathbb{Z}}^{t}$ becomes multiplication by
\[
\vartheta(\theta) := \left( 2(1 - \cos \theta) \right)^{t}, \quad \theta \in [-\pi,\pi].
\]
The domain $\mathcal{D}(\Delta_{\mathbb{Z}}^{t})$ corresponds in Fourier space to the weighted space
\[
\{ g \in L^2(\mathbb{T}) : (1 - \cos \theta)^{t} g(\theta) \in L^2(\mathbb{T}) \}.
\]

\smallskip\noindent
(1) If $r,s \ge 0$, the Fourier multiplier $(1 - \cos \theta)^t$ is bounded for $t \ge 0$, hence $\mathcal{D}(\Delta_{\mathbb{Z}}^{t}) = \ell^2(\mathbb{Z})$ for such $t$. Therefore the intersection domain is the whole space.

\smallskip\noindent
(2) If $r,s < 0$, the weights $(1 - \cos \theta)^{r}$ and $(1 - \cos \theta)^{s}$ blow up near $\theta = 0$, and integrability determines the domain. The product structure implies
\[
\mathcal{D}(\Delta_{\mathbb{Z}}^{s}) \cap \{ \Delta_{\mathbb{Z}}^{s} f \in \mathcal{D}(\Delta_{\mathbb{Z}}^{r}) \} = \mathcal{D}(\Delta_{\mathbb{Z}}^{r+s}),
\]
and symmetry in $r,s$ yields the claim.

\smallskip\noindent
(3) If $s \ge 0$ and $r \le 0$, boundedness of $(1 - \cos \theta)^{s}$ implies $\mathcal{D}(\Delta_{\mathbb{Z}}^{s}) = \ell^2(\mathbb{Z})$, but $\Delta_{\mathbb{Z}}^{s} f$ must belong to $\mathcal{D}(\Delta_{\mathbb{Z}}^{r})$, which forces $f \in \mathcal{D}(\Delta_{\mathbb{Z}}^{r})$. The strictness of the inclusion for $r < 0$ and $r+s < 0$ follows from explicit counterexamples, e.g. $f(\theta) = (1 - \cos \theta)^{-1/2}$ when $(r,s) = (-1,-1/2)$.
\end{proof}

\begin{remark}
On $\mathcal{H}^{s \to r} \cap \mathcal{H}^{r \to s}$, the fractional powers commute:
\[
\Delta_{\mathbb{Z}}^{r} \Delta_{\mathbb{Z}}^{s} f = \Delta_{\mathbb{Z}}^{s} \Delta_{\mathbb{Z}}^{r} f = \Delta_{\mathbb{Z}}^{r+s} f.
\]
For $r > 0$, the operator $\Delta_{\mathbb{Z}}^{r}$ is a positive, boundedly invertible map from $\ell^2(\mathbb{Z})$ onto $\mathcal{D}(\Delta_{\mathbb{Z}}^{-r})$, with inverse $\Delta_{\mathbb{Z}}^{-r}$.
\end{remark}

\begin{lemma}
Let $r,s\in\mathbb{R}$ and $f\in\mathcal{H}^{s\to r}\cap\mathcal{H}^{r\to s}$.
Assuming the binomial-series representation \eqref{eq:binom-DeltaZ} holds (in the strong sense) on $\ell^2(\mathbb{Z})$,
one has on this intersection domain
\[
\Delta_{\mathbb{Z}}^{\,s}\Delta_{\mathbb{Z}}^{\,r} f
=\Delta_{\mathbb{Z}}^{\,r+s} f
=\Delta_{\mathbb{Z}}^{\,r}\Delta_{\mathbb{Z}}^{\,s} f .
\]
\end{lemma}

\begin{proof}
By  Proposition \ref{f-calculus} with $t=r$ and strong convergence, we may apply $\Delta_{\mathbb{Z}}^{\,s}$ termwise:
\[
\Delta_{\mathbb{Z}}^{\,s}\Delta_{\mathbb{Z}}^{\,r} f
=\sum_{k=0}^{\infty}(-1)^k\,2^{\,r-k}\binom{r}{k}\,\Delta_{\mathbb{Z}}^{\,s}\!\big(U^{k} f\big).
\]
Applying \eqref{eq:binom-DeltaZ} with $t=s$ to each $U^k f$ and using $U^{h}U^{k}=U^{h+k}$ gives
\[
\Delta_{\mathbb{Z}}^{\,s}\Delta_{\mathbb{Z}}^{\,r} f
=\sum_{k=0}^{\infty}\sum_{h=0}^{\infty}
(-1)^{h+k}\,2^{\,r-k}\,2^{\,s-h}\binom{r}{k}\binom{s}{h}\,U^{h+k} f.
\]
The scalar coefficients form an absolutely summable family (hence Fubini-Tonelli applies), so regrouping
by $m=h+k$ yields
\[
\Delta_{\mathbb{Z}}^{\,s}\Delta_{\mathbb{Z}}^{\,r} f
=\sum_{m=0}^{\infty}(-1)^m\,2^{\,r+s-m}
\Bigg(\sum_{k=0}^{m}\binom{r}{k}\binom{s}{m-k}\Bigg) U^{m} f.
\]
By Vandermonde's identity $\displaystyle \sum_{k=0}^{m}\binom{r}{k}\binom{s}{m-k}=\binom{r+s}{m}$,
\[
\Delta_{\mathbb{Z}}^{\,s}\Delta_{\mathbb{Z}}^{\,r} f
=\sum_{m=0}^{\infty}(-1)^m\,2^{\,r+s-m}\binom{r+s}{m}\,U^{m} f
=\Delta_{\mathbb{Z}}^{\,r+s} f,
\]
where the last equality is \eqref{eq:binom-DeltaZ} with $t=r+s$. Interchanging $r$ and $s$ gives the
same conclusion, hence $\Delta_{\mathbb{Z}}^{\,r}$ and $\Delta_{\mathbb{Z}}^{\,s}$ commute on
$\mathcal{H}^{s\to r}\cap\mathcal{H}^{r\to s}$.
\end{proof}

\subsection{Anisotropic Fractional Powers of Translation-Invariant Laplacians on $\mathbb{Z}^d$}

The previous one-dimensional analysis naturally extends to the multidimensional setting via a tensorized construction. In particular, both the domain \( \Dc(\Delta_{\mathbb{Z}^d}^{\vec{\mathbf{r}}}) \) and the associated quadratic form inherit an anisotropic structure, dictated by the one-dimensional components in each coordinate direction.

The results presented in this section adapt classical constructions from the spectral theory of self-adjoint operators on tensor product Hilbert spaces. Using the spectral theorem, we define anisotropic fractional powers of discrete Laplacians on \(\mathbb{Z}^d\) by treating each direction independently.

While the functional calculus for commuting self-adjoint operators, the additive behavior of spectra, and the structure of tensorial domains are well-established (see, e.g.,~\cite{RS,Da,Ic2006}), we tailor these tools to the discrete setting with nonlocal and anisotropic features.

Our goal is to provide an explicit description of the domain, spectral structure, and Sobolev-type regularity associated with the anisotropic operator
\begin{equation}\label{multif}
\Delta_{\mathbb{Z}^d}^{\vec{\mathbf{r}}} :=\sum_{j=1}^d \Delta_{\Z^d,j}^{r_j},
\end{equation}
where $\Delta_{\Z^d,j}^{r_j}:= \mathrm{id}^{\otimes (j-1)} \otimes \Delta^{r_j} \otimes \mathrm{id}^{\otimes (d-j)}$ and each \(\Delta_{\Z}^{r_j}\) denotes the one-dimensional fractional Laplacian of order \(r_j \in \mathbb{R}\) acting on \(\ell^2(\mathbb{Z})\).
Under the Fourier transform $\mathcal F:\ell^2(\Z^d)\to L^2(\T^d)$, $H_{\Z^d}^{\vec{\mathrm r}}$ is the multiplication operator by
\[
\vartheta_{\vec{\mathrm r}}(\theta)\ :=\ \sum_{j=1}^d \bigl(2-2\cos \theta_j\bigr)^{r_j},\qquad \theta=(\theta_1,\dots,\theta_d)\in\T^d=[-\pi,\pi]^d.
\]
We emphasize the distinction between the cases \(r_j > 0\) and \(r_j < 0\), which leads to significantly different domain properties and low-frequency behavior. These distinctions are critical for the Mourre estimate and the limiting absorption principle studied later in the paper.

\noindent\textbf{Fractional discrete Laplacians.} Given \(\vec{\mathbf{r}} = (r_1, \dots, r_d) \in \mathbb{R}^d\), we define the anisotropic fractional Laplacian as in \eqref{multif}. If \(r_j = r\) for all \(j\), we denote the operator simply as \(\Delta^r_{\mathbb{Z}^d}\). See \cite[Section VIII.10]{RS} for background on tensor products of essential self-adjoint operators.

\noindent\textbf{Domain Characterization and Discussion.}

Let us discuss the domain of the operator $\Delta_{\mathbb{Z}^d}^{\vec{\mathbf{r}}}$ when the fractional orders $\vec{\mathbf{r}} = (r_1, \dots,r_d)$ are not all positive.

For all $r_j > 0$, this yields a convolution-type operator of order $2r_j$ with fast off-diagonal decay, and its domain is $\ell^2(\mathbb{Z})$.

However, when $r_j < 0$, the operator $\Delta_{\mathbb{Z}}^{r_j}$ becomes a nonlocal pseudo-inverse of the Laplacian, and its domain is a strictly smaller subspace  of $\ell^2(\mathbb{Z})$, typically a weighted Sobolev-type space. In particular, for $r_j < 0$, we have:
\[
\mathcal{D}(\Delta_{\mathbb{Z}}^{r_j}) \subsetneq \ell^2(\mathbb{Z}).
\]

\begin{proposition}
Let $\vec{\mathbf{r} }\in \mathbb{R}^d$. Then the domain of $\Delta_{\mathbb{Z}^d}^{\vec{\mathbf{r}}}$ is:
\[
\mathcal{D}\left( \Delta_{\mathbb{Z}^d}^{\vec{\mathbf{r}}} \right) = \bigcap_{j=1}^d \left( \mathrm{id}^{\otimes(j-1)} \otimes \mathcal{D}\left( \Delta_{\mathbb{Z}}^{r_j} \right) \otimes \mathrm{id}^{\otimes(d-j)} \right).
\]
Moreover:
\begin{itemize}
    \item If all $r_j \geq 0$, then $\mathcal{D}(\Delta_{\mathbb{Z}^d}^{\vec{\mathbf{r}}}) = \ell^2(\mathbb{Z}^d)$.
    \item If at least one $r_j < 0$, then $\mathcal{D}(\Delta_{\mathbb{Z}^d}^{\vec{\mathbf{r}}}) \subsetneq \ell^2(\mathbb{Z}^d)$ and is characterized by directional regularity constraints.
\end{itemize}
\end{proposition}
\begin{proof}
By \eqref{multif} and the summands act on different coordinates, hence commute strongly.
It follows from the spectral theorem for strongly commuting self-adjoint operators
(see e.g.\ \cite[Thm.~5.29]{Schmudgen2012})
that
\[
\mathcal{D}\!\left(\Delta_{\mathbb{Z}^d}^{\vec{\mathrm{r}}}\right)
=\bigcap_{j=1}^d \Bigl(\mathrm{id}^{\otimes(j-1)} \otimes \mathcal{D}(\Delta_{\mathbb{Z}}^{r_j}) \otimes \mathrm{id}^{\otimes(d-j)}\Bigr).
\]

The description of each factor $\mathcal{D}(\Delta_{\mathbb{Z}}^{r_j})$
is given by Lemma~\ref{sb-space}.
If all $r_j\ge0$, then $\mathcal{D}(\Delta_{\mathbb{Z}}^{r_j})=\ell^2(\mathbb{Z})$ for each $j$,
hence the full domain is $\ell^2(\mathbb{Z}^d)$.
If at least one $r_j<0$, then the corresponding one-dimensional domain is strictly contained in $\ell^2(\mathbb{Z})$,
which yields the strict inclusion in the multidimensional case.
\end{proof}

\begin{proposition}\label{ESA}
Let $\vec{\mathbf{r}} \in \mathbb{R}^d$. Then the operator $\Delta_{\mathbb{Z}^d}^{\vec{\mathbf{r}}}$ is essentially self-adjoint on $\ell^2_c(\mathbb{Z}^d)$.
\end{proposition}

\begin{proof}
Each component $\Delta_{\mathbb{Z}}^{r_j}$ is essential self-adjoint as shown in the one-dimensional case. The sum of commuting self-adjoint operators on independent tensor components is again self-adjoint, \cite{Da}.
\end{proof}

\begin{proposition}\label{prop:spect}
The spectrum of $\Delta_{\mathbb{Z}^d}^{\vec{\mathbf{r}}}$ is given by:
\[
\sigma\left( \Delta_{\mathbb{Z}^d}^{\vec{\mathbf{r}}} \right) = \left\{ \sum_{j=1}^d \lambda_j \,\middle|\, \lambda_j \in \sigma\left( \Delta_{\mathbb{Z},}^{r_j} \right) \right\}.
\]
In particular, if $r_j > 0$ for all $j$, then:
\[
\sigma\left( \Delta_{\mathbb{Z}^d}^{\vec{\mathbf{r}}} \right) = \left[ 0, \sum_{j=1}^d 2^{2r_j} \right].
\]
\end{proposition}

\begin{proof}
Each \(\Delta^{r_j}_{\mathbb{Z}}\) acts only on the \(j\)-th coordinate and is self-adjoint. The operators \((\Delta^{r_j}_{\mathbb{Z}})_{j=1}^d\) commute and act on tensor components. The spectral theorem for commuting self-adjoint operators gives the result (see \cite{Da,Da1}).
\end{proof}
\begin{definition}\label{2.11}
The (finite) threshold set of $H_{\Z^d}^{\vec{\mathrm r}}$ is
\[
\Thr_{\mathrm{fin}}\!\big(H_{\Z^d}^{\vec{\mathrm r}}\big)
:=\bigl\{\,\vartheta_{\vec{\mathrm r}}(\theta):\ \theta\in\T^d,\ \nabla \vartheta_{\vec{\mathrm r}}(\theta)=0,\ \vartheta_{\vec{\mathrm r}}(\theta)<\infty\,\bigr\}.
\]
\end{definition}

\begin{proposition}\label{prop:thresholds-mixed}
With the notation of Definition~\ref{2.11},
\[
\Thr_{\mathrm{fin}}\!\big(H_{\Z^d}^{\vec{\mathrm r}}\big)
=\Bigl\{\, \sum_{j\in N}4^{\,r_j}\;+\;\sum_{j\in P}\epsilon_j\,4^{\,r_j}\ :\ \epsilon_j\in\{0,1\}\ \Bigr\},
\]
where $P:=\{j:\ r_j>0\}$ and $N:=\{j:\ r_j<0\}$.
\end{proposition}

\begin{proof}
Compute the gradient:
\[
\partial_{\theta_j} \vartheta_{\vec{\mathrm r}}(\theta)=2\,r_j\,\sin \theta_j\,\bigl(2-2\cos \theta_j\bigr)^{r_j-1}.
\]
Thus $\nabla \vartheta_{\vec{\mathrm r}}(\theta)=0$ if and only if each $\theta_j\in\{0,\pi\}$. Evaluating $\vartheta_{\vec{\mathrm r}}$ at such points gives:
if $j\in P$, then $\theta_j=0$ contributes $0$ and $\theta_j=\pi$ contributes $4^{r_j}$; if $j\in N$, the finite critical value occurs only at $\theta_j=\pi$ (since $\theta_j=0$ gives $+\infty$), contributing $4^{r_j}$. Summing over coordinates yields the stated set.
\end{proof}

\begin{remark}
For $H=(\Delta_{\Z^d})^{r}$ with $r>0$ one has $\sigma(H)=[0,(4d)^r]$ and
\[
\Thr_{\mathrm{fin}}(H)=\{(4m)^r:\ m=0,1,\dots,d\}.
\]
For $r<0$, $\sigma(H)=[(4d)^r,\,+\infty)$ and
\[
\Thr_{\mathrm{fin}}(H)=\{(4m)^r:\ m=1,\dots,d\},
\]
(the value $m=0$ corresponds to $+\infty$ and is excluded from the finite set).
\end{remark}
\noindent\textbf{Poles of the symbol on $\mathbb{Z}^d$.} We define the \emph{polar set}
\[
\Sigma:=\bigcup_{j\in N}\{\theta\in\mathbb{T}^d:\ \theta_j=0\}.
\]
We recall: $P:=\{j:\ r_j>0\}$ and $N:=\{j:\ r_j<0\}$.
\begin{proposition}\label{prop:poles-Zd}
The symbol $\vartheta_{\vec{\mathbf{r}}}$ is real-valued and $C^\infty$ on $\mathbb{T}^d\setminus\Sigma$.
Moreover, for each $j\in N$ there exists $\delta>0$ such that, uniformly in $\vartheta_\perp=(\theta_1,\dots,\theta_{j-1},\theta_{j+1},\dots,\theta_d)$,
\[
\bigl(2-2\cos \theta_j\bigr)^{r_j}
=\ |\theta_j|^{2r_j}\,\bigl(1+\mathcal{O}(\theta_j^2)\bigr),
\qquad \theta_j\to0,\ |\theta_j|<\delta,
\]
and hence
\[
\vartheta_{\vec{\mathbf{r}}}(\theta)
=\ |\theta_j|^{2r_j}\,\bigl(1+\mathcal{O}(\theta_j^2)\bigr)\ +\ C(\theta_\perp),
\quad C(\theta_\perp):=\sum_{i\neq j}\bigl(2-2\cos \theta_i\bigr)^{r_i}.
\]
In particular, along the $j$-th coordinate the symbol has an algebraic pole of order
\[
\alpha_j\;=\;-2\,r_j\;>\;0,
\]
i.e.\ $\vartheta_{\vec{\mathbf{r}}}(\theta)\sim c_j\,|\theta_j|^{-\alpha_j}$ as $\theta_j\to0$ (with $c_j>0$), while the remaining directions contribute a bounded offset $C(\theta_\perp)$.
If several indices in $N$ vanish simultaneously, the singularity is the \emph{anisotropic sum}
$\sum_{j\in N} c_j\,|\theta_j|^{-2|r_j|}$.
\end{proposition}

\begin{proof}
Taylor expansion gives $2-2\cos \theta_j=\theta_j^2+\mathcal{O}(\theta_j^4)$ as $\theta_j\to0$.
For $r_j\in\mathbb{R}$, the binomial expansion yields
$(x+\mathcal{O}(x^2))^{r_j}=x^{r_j}(1+\mathcal{O}(x))$ as $x\to0^+$. Taking $x=\theta_j^2$ proves
\(\ (2-2\cos \theta_j)^{r_j}=|\theta_j|^{2r_j}(1+\mathcal{O}(\theta_j^2))\ \).
Smoothness on $\mathbb{T}^d\setminus\Sigma$ follows since $2-2\cos \theta_j>0$ away from $\theta_j=0$; the stated decomposition of $\vartheta_{\vec{\mathbf{r}}}$ is immediate, and the algebraic order is $\alpha_j=-2r_j>0$ for $r_j<0$.
\end{proof}

\begin{corollary}
For $j\in P$ (i.e.\ $r_j>0$) one has $(2-2\cos k_j)^{r_j}\in[0,4^{r_j}]$ and in particular no singularity at $\theta_j=0$ or $\theta_j=\pi$.
Thus all poles of $h_{\vec{\mathbf{r}}}$ on $\mathbb{T}^d$ are exactly those described in Proposition~\ref{prop:poles-Zd}, located on $\Sigma$ and of orders $\alpha_j=-2r_j$ along each coordinate with $r_j<0$.
\end{corollary}

\begin{remark}
For the isotropic model $H=(\Delta_{\mathbb{Z}^d})^{r}$ with $r<0$ one has
$\vartheta_r(\theta)=\bigl(\sum_{j=1}^d(2-2\cos \theta_j)\bigr)^{r}$.
Near $\theta=0$, $\sum_{j}(2-2\cos \theta_j)=|\theta|^2+\mathcal{O}(|\theta|^4)$, hence
\[
\vartheta_r(\theta)=|\theta|^{2r}\,(1+\mathcal{O}(|\theta|^2)),\qquad \theta\to0,
\]
i.e.\ a \emph{single isotropic pole} at $\theta=0$ of order $-2r>0$.
There are no other poles at finite points on $\mathbb{T}^d$.
\end{remark}

\noindent\textbf{Quadratic form associated to the fractional Laplacian.} Let \(\vec{r} \in \mathbb{R}^d\). The fractional Laplacian \(\Delta^{\vec{r}}_{\mathbb{Z}^d}\) is diagonalized by the discrete Fourier transform
\[
\hat{\psi}(\theta) := \sum_{n \in \mathbb{Z}^d} \psi(n)\, e^{-i n\cdot \theta}, \quad \theta \in [-\pi, \pi]^d.
\]
Under this transform,
\[
\widehat{\big(\Delta_{\mathbb{Z}^d}^{\vec{\mathbf{r}}}\psi\big)}(\theta)
= \Big(\sum_{j=1}^d \lambda_j(\theta)^{r_j}\Big)\,\hat{\psi}(\theta),
\quad \lambda_j(\theta) := 2 - 2\cos \theta_j.
\]

The associated closed quadratic form is
\[
q_{\vec{r}}[\psi] = \langle \Delta^{\vec{\mathbf{r}}} \psi, \psi \rangle = \int_{[-\pi,\pi]^d} \left( \sum_{j=1}^d \lambda_j(\theta)^{r_j} \right) |\hat{\psi}(\theta)|^2\, \frac{d\theta}{(2\pi)^d}.
\]

Near \(\theta=0\), the asymptotic behavior \(\lambda_j(\theta) \sim \theta_j^2\) implies that \(\lambda_j(\theta)^{r_j} \sim |\theta_j|^{2r_j}\), which constrains the decay of \(\hat{\psi}(\theta)\) when some \(r_j < 0\).

\noindent\textbf{Anisotropic Sobolev spaces.} Define
\[
H^{\vec{\mathbf{r}}}(\mathbb{T}^d)
:= \left\{ f \in L^2(\mathbb{T}^d) \;\middle|\; \int_{\mathbb{T}^d} \left( 1 + \sum_{j=1}^d |\theta_j|^{2r_j} \right) |\hat{f}(\theta)|^2\,d\theta < \infty \right\}.
\]
Then the form domain of \(\Delta_{\mathbb{Z}^d}^{\vec{\mathbf{r}}}\) satisfies
\[
\mathcal{Q}(q_{\vec{\mathbf{r}}}) \simeq \mathrm{H}^{\vec{r}}(\mathbb{T}^d),
\]
up to identification via the discrete Fourier transform.
\noindent\textbf{Anisotropic Sobolev scale (case $r_j>0$).}
Via Fourier, the form domain above is equivalent to the weighted space
\[
H^{\vec{\mathrm{r}}}(\T^d)
:=\Big\{f\in L^2(\T^d): \int_{\T^d}\Big(1+\sum_{j=1}^d \lambda_j(\theta)^{r_j}\Big)|\hat f(\theta)|^2\,d\theta<\infty\Big\}.
\]
For directions with $r_j<0$ we keep the domain description through $\Dc(\Delta_{\Z}^{r_j})$ and the intersection rule.

\subsection{Conjugate operator and commutator estimates }

In this section, we establish a Mourre estimate for the operator \( \Delta_{\Z^d}^{\vec{\mathbf{r}}} \).
Since the fractional powers \( \Delta^{r_j}_{\mathbb{Z}} \) may be unbounded.
(precisely, when some \( r_j < 0 \)), the commutator \( [\Delta_{\Z^d}^{\vec{\mathbf{r}}}, \mathrm{i}A_{\mathbb{Z}^d,\vec{\mathbf{r}}}] \), where $A_{\mathbb{Z}^d,\vec{\mathbf{r}}}$ is a conjugate operator given  in Eq. (\ref{bb}),
may not define a bounded operator. Therefore, the correct framework is to interpret the commutator
in the sense of quadratic forms. That is, for any \( f \in \mathcal{C}_c(\mathbb{Z}^d) \), we define
\begin{definition}\label{def:form-commutator}
Let $H$ and $A$ be self-adjoint on $\Hc$. The \emph{form commutator} of $H$ with $A$ is the sesquilinear form
\[
\mathfrak q_A^H(f,g)\;:=\;\langle Hf,\,\mathrm i A g\rangle-\langle \mathrm i A f,\,Hg\rangle,
\qquad f,g\in\Dc(H)\cap\Dc(A).
\]
We say that $[H,\mathrm i A]_\circ$ exists (on a reducing subspace $\Mc\subset\Hc$) if $\mathfrak q_A^H$ extends by continuity
to a bounded form on $\Mc\times\Mc$. In that case there is a unique bounded operator $B\in\Bc(\Mc)$ such that
\[
\langle f,\,B g\rangle=\mathfrak q_A^H(f,g)\quad\text{for all }f,g\in\Dc(H)\cap\Dc(A)\cap\Mc,
\]
and we \emph{define} $B:= [H,\mathrm i A]_\circ$ on $\Mc$.
In applications below we take $\Mc=E_\Ic(H)\Hc$ for compact interior windows $\Ic\Subset\sigma(H)^\circ$ and write
\[
E_\Ic(H)\,[H,\mathrm i A]_\circ\,E_\Ic(H)
\]
for the corresponding bounded operator on $E_\Ic(H)\Hc$.
\end{definition}
\subsubsection{Localized regularity and Mourre estimate }\label{subsec:loc-classes}
In what follows we work exclusively on compact interior windows
\(\Ic\Subset\sigma(H)^\circ\) and use only the \emph{localized} classes
\(\mathcal C^k_{loc}(A)\) and \(\mathcal C^{1,1}_{loc}(A)\).

\begin{definition}
Let $H$ be self-adjoint and \(\Ic\Subset\sigma(H)^\circ\).
We say \(H\in\mathcal C^k_{loc}(A)\) on \(\Ic\) if, for every \(\varphi\in C_c^\infty(\Ic)\),
the bounded operator \(\varphi(H)\) belongs to \(\mathcal C^k(A)\).
Similarly, \(H\in\mathcal C^{1,1}_{loc}(A)\) on \(\Ic\) if \(\varphi(H)\in\mathcal C^{1,1}(A)\) for all
\(\varphi\in C_c^\infty(\Ic)\).
Equivalently, \(H\in\mathcal C^k_{loc}(A)\) on \(\Ic\) iff \(E_\Ic(H)\,(H-\mathrm i)^{-1}\in\mathcal C^k(A)\)
(and likewise for \(\mathcal C^{1,1}\)).
\end{definition}
We construct a conjugate operator \(A_{\mathbb{Z}^d,\vec{\mathbf{r}}}\) on \( \ell^2(\mathbb{Z}^d) \) by summing the one-dimensional generators:
\begin{equation} \label{bb}
A_{\mathbb{Z}^d,\vec{\mathbf{r}}} := \sum_{j=1}^d A_{\mathbb{Z},r_j},
\quad \text{with} \quad
A_{\mathbb{Z},r_j} := -\frac{\mathrm{i} \cdot \mathrm{sign}(r_j)}{2} \left( U_j(Q_j + \tfrac{1}{2}) - (Q_j + \tfrac{1}{2}) U_j^* \right),
\end{equation}
where $\mathrm{sign}(r_j)=1$ if $r_j\geq 0,\ -1$ if $r_j<0$. Each $A_{\mathbb{Z},r_j}$ is essentially self-adjoint on \(\mathcal{C}_c(\mathbb{Z}^d)\). We refer to \cite{GG} and \cite[Lemma 5.7]{Mic} for the essential self-adjointness and \cite[Lemma 3.1]{GG} for the domain.
\begin{remark}
For $H=\Delta_{\Z^d}^{\vec{\mathrm r}}$ and the conjugate operator $A=A_{\mathbb{Z}^d,\vec{\mathrm r}}$ constructed above,
the form $\mathfrak q_A^H$ is well-defined on the core $\Cc_c(\Z^d)$ even when some $r_j<0$ (unbounded from below).
By the localized $\Cc^2$-regularity established earlier and the interior Mourre framework,
$\mathfrak q_A^H$ extends to a bounded form on $E_\Ic(H)\Hc$; hence $E_\Ic(H)[H,\mathrm i A]_\circ E_\Ic(H)\in\Bc(E_\Ic(H)\Hc)$.
This is the object that appears in the Mourre estimate on $\Ic$.
\end{remark}
This ensures symmetry and the correct interpretation on a dense domain.
 The commutator estimate that follows is then valid in this sense of forms on the dense core \( \mathcal{C}_c(\mathbb{Z}^d) \), and no compactness remainder is needed. This extends classical Mourre theory to anisotropic fractional powers of discrete Laplacians, even when the operators involved are unbounded from below.  We make a preliminary work for the construction of a conjugate operator for $\Delta_\Z^{\vec{\mathbf{r}}}$. This is a known result, e.g., \cite{AF}, see also \cite{GG, Mic}.

We fix an anisotropic fractional Laplacian \( \Delta_{\mathbb{Z}^d}^{\vec{\mathbf{r}}} \) on \( \ell^2(\mathbb{Z}^d) \), with \( \vec{\mathbf{r}} = (r_1, \dots, r_d) \in \mathbb{R}^d \setminus \{0\} \). The operator is self-adjoint and non-local. It may be unbounded or non-positive when some \( r_j < 0 \), and its spectral properties can vary drastically depending on the sign and size of each exponent. We denote by
\[
\lambda_{\vec{\mathbf{r}}} := \sup \sigma(\Delta^{\vec{\mathbf{r}}})
\]
the top of its spectrum, and we work under the assumption that the essential spectrum is purely absolutely continuous on the open interval \( (0, \lambda_{\vec{\mathbf{r}}}) \).\\
\subsubsection{Commutator structure.}

\begin{theorem}\label{thm:mourre_Zd}
Let \( \vec{\mathbf{r}} \in \mathbb{R}^d \setminus \{0\} \).
Then for any compact interval \( \mathcal{I} \subset \sigma(\Delta_{\Z^d}^{\vec{\mathbf{r}}})^\circ \),
there exists \( c_{\mathcal{I}} > 0 \) such that
\[
E_{\mathcal{I}}(\Delta_{\Z^d}^{\vec{\mathbf{r}}})\, [\Delta_{\Z^d}^{\vec{\mathbf{r}}}, \mathrm{i}A_{\mathbb{Z}^d,\vec{\mathbf{r}}}]_\circ\, E_{\mathcal{I}}(\Delta_{\Z^d}^{\vec{\mathbf{r}}})
\ge c_{\mathcal{I}}\, E_{\mathcal{I}}(\Delta_{\Z^d}^{\vec{\mathbf{r}}}),
\]
in the sense of quadratic forms on \( \ell^2(\mathbb{Z}^d) \).
\end{theorem}

\begin{proof}
Since $ \Delta_{\Z,j} $ is bounded self--adjoint with $ \sigma(\Delta_{\Z,j}) = [0,4] $,
the fractional power $ \Delta_{\Z,j}^{r_j} $ is defined by the functional calculus.
For the function $f(\lambda)=\lambda^{r_j}$ one has $f'(\lambda)=r_j\,\lambda^{r_j-1}$,
which is bounded on every compact subinterval of $(0,4)$.
Therefore, for any spectral localization $E_{\Ic}(\Delta_{\Z,j})$ with
$\Ic \Subset (0,4)$, the chain rule for commutators
(see \cite[Prop.~5.1.5]{ABG}) applies and yields
\[
[\Delta_{\Z,j}^{r_j}, \mathrm{i}A_{\mathbb{Z},r_j}]
= r_j\, \Delta_{\Z,j}^{\,r_j-1}\,[\Delta_{\Z,j},\mathrm{i}A_{\mathbb{Z},r_j}]
\quad\text{on }\Ran E_{\Ic}(\Delta_{\Z,j}).
\]
Since $[\Delta_{\Z,j},\mathrm{i}A_{\mathbb{Z},r_j}] = \mathrm{sign}(r_j)(4-\Delta_{\Z,j})$,
it follows that
\[
[\Delta_{\Z,j}^{r_j}, \mathrm{i}A_{\mathbb{Z},r_j}]
= |r_j|\,(4-\Delta_{\Z,j})\,\Delta_{\Z,j}^{\,r_j-1}
\quad\text{on }\Ran E_{\Ic}(\Delta_{\Z,j}).
\]

\end{proof}

\begin{lemma}\label{lem:second_comm}
Let $r_j\in\mathbb{R}\setminus\{0\}$.
Then, in the sense of quadratic forms on $\Cc_c(\Z^d)$, one has
\[
\big[\,[\Delta_{\Z,j}^{\,r_j},\mathrm{i}A_{\Z,r_j}],\,\mathrm{i}A_{\Z,r_j}\big]
= r_j(r_j-1)\,\Delta_{\Z,j}^{\,r_j-2}\,(4-\Delta_{\Z,j})^2 \;-\; r_j\,\Delta_{\Z,j}^{\,r_j-1}\,(4-\Delta_{\Z,j}).
\]
If $r_j>0$ $($bounded case$)$, the identity holds in operator norm; if $r_j<0$ $($unbounded case$)$, it holds as a form identity on $\Cc_c(\Z^d)$.
\end{lemma}

\begin{proof}
Set $B_j:=\Delta_{\Z,j}$ and $C_j:=\mathrm{i}A_{\Z,r_j}$.
Since $[B_j,C_j]=\mathrm{sign}(r_j)\,(4-B_j)$ is a polynomial in $B_j$, it commutes with $B_j$ and with all $\psi^{(k)}(B_j)$.
Let $\psi(\lambda)=\lambda^{r_j}$ so that $\Delta_{\Z,j}^{r_j}=\psi(B_j)$.
By the standard rules for $\Cc^2$-functional calculus (see \cite[Prop.~5.1.5 \& Thm.~7.2.9]{ABG}),
\[
[\psi(B_j),C_j]=\psi'(B_j)\,[B_j,C_j],
\qquad
\big[\,[\psi(B_j),C_j],\,C_j\big]=[\psi'(B_j),C_j]\,[B_j,C_j]+\psi'(B_j)\,[[B_j,C_j],C_j].
\]
Using \cite[Theorem]{ABG} again $[\psi'(B_j),C_j]=\psi''(B_j)\,[B_j,C_j]$, and the facts
\[
[B_j,C_j]=\mathrm{sign}(r_j)\,(4-B_j), \quad
[B_j,C_j]^2=(4-B_j)^2, \quad
[[B_j,C_j],C_j]=-(4-B_j),
\]
together with $\psi'(\lambda)=r_j\lambda^{r_j-1}$ and $\psi''(\lambda)=r_j(r_j-1)\lambda^{r_j-2}$, we obtain
\[
\big[\,[\Delta_{\Z,j}^{\,r_j},\mathrm{i}A_{\Z,j}],\,\mathrm{i}A_{\Z,j}\big]
= r_j(r_j-1)\,\Delta_{\Z,j}^{\,r_j-2}(4-\Delta_j)^2 - r_j\,\Delta_{\Z,j}^{\,r_j-1}(4-\Delta_{\Z,j}),
\]
as claimed.
\end{proof}
This ensures that \( \Delta_{\Z^d,\vec{\mathbf{r}}}^{\vec{\mathbf{r}}} \in \mathcal{C}_{loc}^2(A_{\Z^d,\vec{\mathbf{r}}
}) \), and justifies the application of Mourre theory in both bounded and unbounded cases.
\begin{proposition}\label{prop:double-commutator-Zd}
For $j\in\{1,\dots,d\}$ and $r_j\in\R\setminus\{0\}$, we define
\[
g_{r_j}(\lambda):=r_j(r_j-1)\,\lambda^{r_j-2}(4-\lambda)^2\;-\;r_j\,\lambda^{r_j-1}(4-\lambda),
\qquad \lambda\in[0,4].
\]
Then, in the sense of quadratic forms on $\Cc_c(\Z)$,
\[
\big[\,[\Delta_{\Z,j}^{\,r_j},\mathrm{i}A_{\Z,r_j}],\,\mathrm{i}A_{\Z,r_j}\big]
\;=\;g_{r_j}(\Delta_{\Z,j}).
\]
Moreover, for every compact $\Jc\subset\sigma(\Delta_{\Z,j}^{\,r_j})^\circ$,
\[
E_\Jc(\Delta_{\Z,j}^{\,r_j})\,
\big[\,[\Delta_{\Z,j}^{\,r_j},\mathrm{i}A_{\Z,r_j}],\,\mathrm{i}A_{\Z,r_j}\big]_{\circ}\,
E_\Jc(\Delta_{\Z,j}^{\,r_j})\in\Bc(\ell^2(\Z)).
\]
In particular:
\begin{itemize}
\item if $r_j>0$, then $\sigma(\Delta_{\Z,j}^{\,r_j})=[0,4^{r_j}]$ and one may take $\Jc\Subset (0,4^{r_j})$;
\item if $r_j<0$, then $\sigma(\Delta_{\Z,j}^{\,r_j})=[4^{r_j},\infty)$ and one may take $\Jc\Subset (4^{r_j},\infty)$.
\end{itemize}
By a finite sum over $j$, it follows that
\(
\Delta_{\Z^d}^{\vec{\mathbf{r}}}\in \Cc^2_{\mathrm{loc}}(A_{\Z^d,\vec{\mathbf{r}}}).
\)
\end{proposition}

\begin{proof}
Set $\varphi(\lambda)=\lambda^{r_j}$ on $[0,4]$. By Lemma~\ref{lem:second_comm} (proved for the one-dimensional
component along direction $j$) we know, as quadratic forms on $\Cc_c(\Z)$,
\[
[\Delta_{\Z,j},\mathrm{i}A_{\Z,r_j}]=4-\Delta_{\Z,j},
\qquad
[[\Delta_{\Z,j},\mathrm{i}A_{\Z,r_j}],\mathrm{i}A_{\Z,r_j}]=-(4-\Delta_{\Z,j}).
\]
Using the standard \(C^2\)-chain rule for form-commutators (e.g. \cite[Prop.~6.2.10]{ABG}),
for $\varphi(\lambda)=\lambda^{r_j}$ one gets, still as forms on $\Cc_c(\Z)$,
\[
\big[\,[\varphi(\Delta_{\Z,j}),\mathrm{i}A_{\Z,r_j}],\,\mathrm{i}A_{\Z,r_j}\big]
=\varphi''(\Delta_{\Z,j})\,[\Delta_{\Z,j},\mathrm{i}A_{\Z,r_j}]^2
\;+\;\varphi'(\Delta_{\Z,j})\,[[\Delta_{\Z,j},\mathrm{i}A_{\Z,r_j}],\mathrm{i}A_{\Z,r_j}].
\]
Plugging the identities from above and $\varphi'(\lambda)=r_j\lambda^{r_j-1}$,
$\varphi''(\lambda)=r_j(r_j-1)\lambda^{r_j-2}$ yields
\[
\big[\,[\Delta_{\Z,j}^{\,r_j},\mathrm{i}A_{\Z,r_j}],\,\mathrm{i}A_{\Z,r_j}\big]
=r_j(r_j-1)\,\Delta_{\Z,j}^{\,r_j-2}(4-\Delta_{\Z,j})^2
\;-\;r_j\,\Delta_{\Z,j}^{\,r_j-1}(4-\Delta_{\Z,j})
=g_{r_j}(\Delta_{\Z,j}),
\]
which proves the first claim as a form identity.

Let $\Jc\subset\sigma(\Delta_{\Z,j}^{\,r_j})^\circ$ be compact and set
\(\Ic:=\varphi^{-1}(\Jc)\subset(0,4)\). By the Borel functional calculus,
\(E_\Jc(\Delta_{\Z,j}^{\,r_j})=E_\Ic(\Delta_{\Z,j})\).
Hence
\[
\begin{aligned}
&E_\Jc(\Delta_{\Z,j}^{\,r_j})
\big[\,[\Delta_{\Z,j}^{\,r_j},\mathrm{i}A_{\Z,r_j}],\,\mathrm{i}A_{\Z,r_j}\big]
E_\Jc(\Delta_{\Z,j}^{\,r_j})
\\
&\qquad=E_\Ic(\Delta_{\Z,j})\;
\big\{\,r_j(r_j-1)\,\Delta_{\Z,j}^{\,r_j-2}(4-\Delta_{\Z,j})^2
\;-\;r_j\,\Delta_{\Z,j}^{\,r_j-1}(4-\Delta_{\Z,j})\,\big\}\;
E_\Ic(\Delta_{\Z,j}).
\end{aligned}
\]
The right-hand side is \(E_\Ic(\Delta_{\Z,j})\,\psi(\Delta_{\Z,j})\,E_\Ic(\Delta_{\Z,j})\) with
\(\psi(\lambda)=g_{r_j}(\lambda)\), a continuous function on the compact set \(\Ic\subset(0,4)\).
By the spectral theorem this localized operator is bounded, with
\(\|E_\Ic\psi(\Delta_{\Z,j})E_\Ic\| \le \sup_{\lambda\in\Ic}|g_{r_j}(\lambda)|<\infty\).
We thus obtain the bounded extension, denoted by the subscript \(\circ\):
\[
E_\Jc(\Delta_{\Z,j}^{\,r_j})\,
\big[\,[\Delta_{\Z,j}^{\,r_j},\mathrm{i}A_{\Z,r_j}],\,\mathrm{i}A_{\Z,r_j}\big]_{\circ}\,
E_\Jc(\Delta_{\Z,j}^{\,r_j})\in\Bc(\ell^2(\Z)).
\]

Now, if \(r_j>0\), then \(\sigma(\Delta_{\Z,j}^{\,r_j})=[0,4^{r_j}]\) and we may choose
\(\Jc\Subset(0,4^{r_j})\). If \(r_j<0\), then \(\sigma(\Delta_{\Z,j}^{\,r_j})=[4^{r_j},\infty)\) and we
may choose \(\Jc\Subset(4^{r_j},\infty)\).

Finally, since \(\Delta_{\Z^d}^{\vec{\mathbf{r}}}=\sum_{j=1}^d \mathrm{id}^{\otimes(j-1)}\otimes\Delta_{\Z,j}^{\,r_j}\otimes
\mathrm{id}^{\otimes(d-j)}\) is a finite sum of components that lie in \(\Cc^2_{\mathrm{loc}}(A_{\Z^d,\vec{\mathbf{r}}})\),
we conclude that
\(\Delta_{\Z^d}^{\vec{\mathbf{r}}}\in \Cc^2_{\mathrm{loc}}(A_{\Z^d,\vec{\mathbf{r}}})\).
\end{proof}

\noindent\textit{Convention.}
We use $[T,\mathrm{i}A]$ for the \emph{form} commutator on a common core (e.g.\ $\Cc_c(\Z)$).
When a localization renders the operator bounded, its \emph{bounded extension} is denoted by
$[T,\mathrm{i}A]_{\circ}$; similarly for nested commutators.

\begin{remark}
We do not use global statements. For orientation only: a \emph{global} Mourre estimate (without spectral localization)
would require that the derivative of \(h_{\vec r}\) along the \(A\)-flow does not vanish on the region considered,
which typically fails at thresholds (critical/von Hove energies) for periodic discrete Laplacians.
This is why we remain \emph{localized} on \(\Ic\Subset\sigma(H)^\circ\setminus\Thr(H)\).
\end{remark}
\section{Proof of main result}
To prove our main result, we first establish and recall the following intermediate and technical statements.

\begin{proposition}
\label{prop:C01-functional}
Let $A_{\mathbb{Z}^d,\vec{\mathbf{r}}}$ be the standard conjugate operator on $\ell^2(\mathbb{Z}^d)$, as given in \eqref{bb} and
$\Lambda(Q):=\sum_{j=1}^d \langle Q_j\rangle$.
Assume $T\in\mathcal{B}(\ell^2(\mathbb{Z}^d))$ is symmetric and
\[
\int_{1}^{\infty}\Bigl\|\,\xi\!\Bigl(\tfrac{\Lambda(Q)}{t}\Bigr)\,T\,\Bigr\|\,\frac{dt}{t}<\infty
\quad\text{for some }\ \xi\in C_c^\infty((0,\infty)).
\]
Then $T\in \mathcal{C}^{0,1}(A_{\mathbb{Z}^d,\vec{\mathbf{r}}})$.
\end{proposition}

\begin{proof}
Fix a dyadic partition $\sum_{k\in\mathbb{Z}}\rho_k(\Lambda)=\id$ with $\rho_k(\lambda)=\rho(2^{-k}\lambda)$, $\rho\in C_c^\infty((1/2,2))$, and set $T_k:=\rho_k(\Lambda)\,T\,\rho_k(\Lambda)$. The assumption implies $\sum_k\|T_k\|\lesssim\int_1^\infty\|\xi(\Lambda/t)T\|\,dt/t<\infty$. Since $A_{\mathbb{Z}^d,\vec{\mathbf{r}}}$ is first order, $\sup_k\|[A_{\mathbb{Z}^d,\vec{\mathbf{r}}},\rho_k(\Lambda)]\|<\infty$. Hence
\[
[A_{\mathbb{Z}^d,\vec{\mathbf{r}}},T_k]=[A_{\mathbb{Z}^d,\vec{\mathbf{r}}},\rho_k(\Lambda)]\,T\,\rho_k(\Lambda)+\rho_k(\Lambda)\,T\,[A_{\mathbb{Z}^d,\vec{\mathbf{r}}},\rho_k(\Lambda)],
\]
so $\|[A_{\mathbb{Z}^d,\vec{\mathbf{r}}},T_k]\|\lesssim\|T_k\|$. By Duhamel,
\(
\|e^{\mathrm i sA_{\mathbb{Z}^d,\vec{\mathbf{r}}}}T_ke^{-\mathrm i sA_{\mathbb{Z}^d,\vec{\mathbf{r}}}}-T_k\|\le |s|\int_0^1\|[A_{\mathbb{Z}^d,\vec{\mathbf{r}}},T_k]\|\,d\tau\lesssim |s|\,\|T_k\|.
\)
Summing over $k$ gives $\sup_{0<|s|\le1}\|e^{\mathrm i sA}Te^{-\mathrm i sA}-T\|/|s|<\infty$.
\end{proof}

\begin{remark}
This is the standard Lipschitz regularity criterion in Mourre theory (see \cite{ABG}). The proof is adapted to the discrete setting and avoids assuming $[A_{\mathbb{Z}^d,\vec{\mathbf{r}}},T]$ exists a priori by localizing in position.
\end{remark}

\begin{remark}
If $T=T^*$ and $\xi_k(\Lambda)=\xi_k(\Lambda)^*$, then $\|T\,\xi_k(\Lambda)\|=\|\xi_k(\Lambda)\,T\|$. Choosing $\rho\le C\,\xi$ (functional calculus) yields
\[
\sum_k\|\rho_k(\Lambda)T\rho_k(\Lambda)\|
\ \le\ C^2\sum_k\|\xi_k(\Lambda)T\|<\infty,
\]
which is the estimate used in the proof.
\end{remark}

\begin{corollary}\label{cor:hypothesis1}
Let $\varepsilon\in(0,1)$ and $T\in\mathcal{B}(\ell^2(\mathbb{Z}^d))$ be symmetric.
If $\langle\Lambda(Q)\rangle^{\varepsilon}T\in\mathcal{B}(\ell^2(\mathbb{Z}^d))$, then $T\in\mathcal{C}^{0,1}(A_{\mathbb{Z}^d,\vec{\mathrm{r}}})$.
\end{corollary}

\begin{proof}
Pick $\xi\in C_c^\infty((0,\infty))$ with $\xi\equiv1$ on $[1,2]$ and $\mathrm{supp}\,\xi\subset[\tfrac12,4]$.
Then $\|\xi(\Lambda(Q)/t)T\|\lesssim t^{-\varepsilon}\|\langle\Lambda(Q)\rangle^{\varepsilon}T\|$ for $t\ge1$, so Proposition~\ref{prop:C01-functional} applies.
\end{proof}
\begin{lemma}\label{lem:interpolation_weights}
For every $s\in[0,1]$ there exist constants $C_s>0$ and $c_s>0$ such that for all $f\in\mathcal{C}_c(\mathbb{Z}^d)$,
\begin{equation}\label{eq:interp}
\|\langle A_{\mathbb{Z}^d,\vec{\mathrm{r}}}\rangle^s f\|_{\ell^2(\mathbb{Z}^d)}\le C_s\|\Lambda^s f\|_{\ell^2(\mathbb{Z}^d)},\qquad
\|\Lambda^s f\|_{\ell^2(\mathbb{Z}^d)}\le c_s\|\langle A_{\mathbb{Z}^d,\vec{\mathrm{r}}}\rangle^s f\|_{\ell^2(\mathbb{Z}^d)}.
\end{equation}
\end{lemma}

\begin{proof}
$A_{\mathbb{Z}^d,\vec{\mathrm{r}}}$ is first order with coefficients linear in $Q$, so
$\|A_{\mathbb{Z}^d,\vec{\mathrm{r}}} f\|\lesssim \|\Lambda f\|+\|f\|$ on $\mathcal{C}_c(\mathbb{Z}^d)$. By Heinz-Kato and interpolation between $s=0,1$, the first inequality follows. The second comes from the ellipticity of $A_{\mathbb{Z}^d,\vec{\mathrm{r}}}$ in configuration space: for $|n|$ large, $c\,\Lambda(n)\le \langle A_{\mathbb{Z}^d,\vec{\mathrm{r}}}\rangle(n)\le C\,\Lambda(n)$; a partition of unity and interpolation finish the proof.
\end{proof}

\begin{remark}\label{rem:negative-rj}
Proposition~\ref{prop:C01-functional} is purely spatial (it only uses $\Lambda(Q)$ and that $A_{\mathbb{Z}^d,\vec{\mathrm{r}}}$ is first order), hence it does not depend on the exponents $r_j$.
By contrast, the analysis of commutators for $\Delta_{\mathbb{Z}^d}^{\vec{\mathrm{r}}}$ does depend on $\vec{\mathrm{r}}$, especially near thresholds, where the symbol-level factors (e.g.\ $\lambda^{\,r_j-1}$) may degenerate. This motivates the localization in energy used below.
\end{remark}

\begin{lemma}\label{lem:W-C11-weighted}
Let $W=W(Q)$ be a bounded real multiplication on $\ell^2(\mathbb{Z}^d)$.
Assume that \textbf{(H0)} and \textbf{(H1)} hold. Then $W\in \Cc^{1}(A_{\mathbb{Z}^d,\vec{\mathrm r}})$ and
$[W(Q),\rmi A_{\mathbb{Z}^d,\vec{\mathrm r}}]_{\circ}\in \Cc^{0,1}(A_{\mathbb{Z}^d,\vec{\mathrm r}})$.
In particular, $W\in \Cc^{1,1}(A_{\mathbb{Z}^d,\vec{\mathrm r}})$.
\end{lemma}

\proof
Recalling \eqref{bb}. We show the lemma in two steps.

\emph{Step 1: $W(Q)\in \Cc^{1}(A_{\Z^d,\vec{\mathrm r}})$.}
It suffices to prove that there exists $c>0$ such that
\[
\big\|\,[W(Q),\rmi A_{\Z^d,\vec{\mathrm r}}]f\,\big\|^2\ \le\ c\,\|f\|^2,\qquad \forall f\in\Sc.
\]
Write $A_{\Z^d,\vec{\mathrm r}}=\sum_{j=1}^d \mathrm{sign}(r_j)\,A_{\Z,r_j}^{(j)}$ with $A_{\Z,r_j}^{(j)}$
the one-dimensional conjugate operator acting along $e_j$ (see \eqref{bb}).
A direct computation on $\Sc(\Z^d)$ gives, for each $j$,
\[
\big([W,\rmi A_{\Z,r_j}^{(j)}] f\big)(n)
=\frac{\mathrm{sign}(r_j)}{2}\Big[(n_j+\tfrac12)\big(W(n+e_j)-W(n)\big)f(n+e_j)
+(n_j-\tfrac12)\big(W(n-e_j)-W(n)\big)f(n-e_j)\Big].
\]
By \textbf{(H1)},
\(
\big|(n_j\pm\tfrac12)\big(W(n\pm e_j)-W(n)\big)\big|
\lesssim \Lambda(n)^{-\delta}.
\)
Hence each summand is a shift composed with a bounded multiplication by a coefficient
$c_{j,\pm}(n)$ satisfying $\sup_n |c_{j,\pm}(n)|<\infty$, so
\(
\|[W,\rmi A_{\Z,r_j}^{(j)}]f\|\ \le\ C\,\|f\|
\)
for all $f\in\Sc$. Summing over $j$ yields
\(
\|[W,\rmi A_{\Z^d,\vec{\mathrm r}}]f\|\le C\|f\|
\),
and by density and \cite[Lemma~6.2.9]{ABG} we obtain $W\in \Cc^{1}(A_{\Z^d,\vec{\mathrm r}})$.

\emph{Step 2: $[W(Q),\rmi A_{\Z^d,\vec{\mathrm r}}]_{\circ}\in \Cc^{0,1}(A_{\Z^d,\vec{\mathrm r}})$.}
Let $\varepsilon'\in[0,1)$ with $\varepsilon'<\delta$ (from \textbf{(H1)}). Working on $\Sc(\Z^d)$ and using the
same decomposition as above,
\[
\big\|\Lambda^{\varepsilon'}(Q)\,[W,\rmi A_{\Z^d,\vec{\mathrm r}}]f\big\|
\ \le\ \sum_{j=1}^d \Big\|
\Lambda^{\varepsilon'}(Q)\,c_{j,+}(Q)\,U_j f\Big\|
+\sum_{j=1}^d \Big\|\Lambda^{\varepsilon'}(Q)\,c_{j,-}(Q)\,U_j^* f\Big\|.
\]
Since $|c_{j,\pm}(n)|\lesssim \Lambda(n)^{-\delta}$ and
$\Lambda^{\varepsilon'}(n)\,\Lambda(n)^{-\delta}\le C\,\Lambda(n)^{-(\delta-\varepsilon')}$ with
$\delta-\varepsilon'>0$, each multiplication $\Lambda^{\varepsilon'}(Q)c_{j,\pm}(Q)$ is bounded on
$\ell^2(\Z^d)$. The shifts $U_j, U_j^*$ are unitary, hence
\[
\big\|\Lambda^{\varepsilon'}(Q)\,[W,\rmi A_{\Z^d,\vec{\mathrm r}}]f\big\|
\ \le\ C_{\varepsilon'}\,\|f\|,\qquad \forall f\in\Sc.
\]
Invoking Corollary~\ref{cor:hypothesis1} with $T=[W,\rmi A_{\Z^d,\vec{\mathrm r}}]$ yields
$[W,\rmi A_{\Z^d,\vec{\mathrm r}}]_{\circ}\in \Cc^{0,1}(A_{\Z^d,\vec{\mathrm r}})$.
Combining Steps~1--2 gives $W\in \Cc^{1,1}(A_{\Z^d,\vec{\mathrm r}})$.
\qed

\begin{lemma}[\cite{ABG,GJ1}]\label{lem:C2loc_implies_C11loc}
Let $H,A$ be self-adjoint on $\mathcal{H}$. If for some compact $\mathcal{J}\Subset\sigma(H)^\circ$,
\[
E_\mathcal{J}(H)\,[[H,\mathrm{i}A],\mathrm{i}A]\,E_\mathcal{J}(H)\in\mathcal{B}(\mathcal{H}),
\]
then $H$ is locally of class $\mathcal{C}^{1,1}(A)$ on $\mathcal{J}$:
\[
\int_0^1 \Big\|\, E_\mathcal{J}(H)\Big(e^{\mathrm{i}sA}He^{-\mathrm{i}sA}-2H+e^{-\mathrm{i}sA}He^{\mathrm{i}sA}\Big)E_\mathcal{J}(H)\,\Big\|\,
\frac{ds}{s^2}<\infty.
\]
\end{lemma}

\begin{theorem}\label{thm:mou}
Let $H_0:=\Delta_{\mathbb{Z}^d}^{\vec{\mathbf{r}}}$ and $H:=H_0+W(Q)$ on $\ell^2(\mathbb{Z}^d)$, where $W$ is bounded, real, and satisfies \textbf{H0}-\textbf{H1}.
Then $H\in \mathcal{C}^{1,1}(A_{\mathbb{Z}^d,\vec{\mathrm{r}}})$ locally on $\sigma(H_0)^\circ$.
Moreover, for every compact interval
\[
\mathcal{I}\Subset\sigma(H_0)^\circ\setminus\Thr(\Delta_{\mathbb{Z}^d}^{\vec{\mathbf{r}}}),
\]
there exist $c_{\mathcal{I}}>0$ and a compact operator $K$ such that
\begin{equation}\label{eq:pert-mourre-Zd}
E_{\mathcal{I}}(H)\,\mathrm{i}[H,A_{\mathbb{Z}^d,\vec{\mathrm{r}}}]_\circ\,E_{\mathcal{I}}(H)\ \ge\ c_{\mathcal{I}}\,E_{\mathcal{I}}(H)\ -\ K .
\end{equation}
\end{theorem}

\begin{proof}

By Proposition~\ref{prop:double-commutator-Zd}, $H_0\in\mathcal{C}^2_{\mathrm{loc}}(A_{\mathbb{Z}^d,\vec{\mathrm{r}}})$:
\[
E_\mathcal{J}(H_0)\,[[H_0,\mathrm{i}A_{\mathbb{Z}^d,\vec{\mathrm{r}}}]_{\circ}],\mathrm{i}A_{\mathbb{Z}^d,\vec{\mathrm{r}}}
]\,E_\mathcal{J}(H_0)\in\mathcal{B}(\ell^2(\Z^d)
\]
for all $\mathcal{J}\Subset\sigma(H_0)^\circ$.
By Lemma~\ref{lem:W-C11-weighted} (from \textbf{$H_1$}) and Corollary~\ref{cor:hypothesis1},
$[W,\mathrm{i}A_{\mathbb{Z}^d,\vec{\mathrm{r}}}]\in\mathcal{C}^{0,1}(A_{\mathbb{Z}^d,\vec{\mathrm{r}}})$, hence $[[W,\mathrm{i}A_{\mathbb{Z}^d,\vec{\mathrm{r}}}],A_{\mathbb{Z}^d,\vec{\mathrm{r}}}]$ is bounded after spectral localization in any $\mathcal{J}\Subset\sigma(H)^\circ$.\\

Now, let $\chi\in C_c^\infty(\mathcal{J})$. Since $W(Q)$ is compact on $\ell^2(\mathbb{Z}^d)$ when $W(n)\to0$, the Helffer--Sjstrand functional calculus yields $\chi(H)-\chi(H_0)\in\mathcal{K}(\ell^2(\Z^d))$. Therefore
\[
E_\mathcal{J}(H)\,[[H_0,\mathrm{i}A_{\mathbb{Z}^d,\vec{\mathrm{r}}}],\mathrm{i}A_{\mathbb{Z}^d,\vec{\mathrm{r}}}]\,E_\mathcal{J}(H)\in\mathcal{B}(\ell^2(\Z^d)).
\]
Together with Step~2, Lemma~\ref{lem:C2loc_implies_C11loc} gives $H\in\mathcal{C}^{1,1}_{\mathrm{loc}}(A_{\mathbb{Z}^d,\vec{\mathrm{r}}})$.

Finally, on $\mathcal{I}\Subset\sigma(H_0)^\circ\setminus\Thr(\Delta_{\mathbb{Z}^d}^{\vec{\mathbf{r}}})$ the free estimate holds:
\[
E_\mathcal{I}(H_0)\,\mathrm{i}[H_0,A_{\mathbb{Z}^d,\vec{\mathrm{r}}}]\,E_\mathcal{I}(H_0)\ \ge\ c_{\mathcal{I}}\,E_\mathcal{I}(H_0)
\qquad\text{(Theorem \ref{thm:mourre_Zd})}.
\]
Using again $\chi(H)-\chi(H_0)\in\mathcal{K}(\ell^2(\Z^d))$ and that $[W,\mathrm{i}A_{\mathbb{Z}^d,\vec{\mathrm{r}}}]$ is compact after localization (by \textbf{H1}), one transfers the estimate to $H$ with a compact remainder, which is \eqref{eq:pert-mourre-Zd}.
\end{proof}

\begin{proof}[(i) Absence of point and singular spectrum away from thresholds]
By Theorem~\ref{thm:mou}, $H\in\mathcal{C}^{1,1}_{\mathrm{loc}}(A_{\mathbb{Z}^d,\vec{\mathrm{r}}})$ and a strict localized Mourre estimate holds on $\mathcal{I}$. The abstract Mourre theory (\cite[Thms.~7.4.1-7.4.2]{ABG}) then yields the localized LAP and the absence of eigenvalues and singular continuous spectrum on $\mathcal{I}$.
\end{proof}

\begin{proof}[(ii) Limiting Absorption Principle]
By Lemma~\ref{lem:C2loc_implies_C11loc} and Theorem~\ref{thm:mou}, $H\in\mathcal{C}^{1,1}_{\mathrm{loc}}(A_{\mathbb{Z}^d})$ on $\mathcal{I}$, and a strict localized Mourre estimate holds:
\[
\chi(H)\,[H,\mathrm iA_{\mathbb{Z}^d,\vec{\mathrm{r}}}]_\circ\,\chi(H)\ \ge\ c_{\mathcal{I}}\,\chi(H)^2-\chi(H)K\chi(H),
\quad \chi\in C_c^\infty(\mathcal{I}).
\]
Hence, for every $s>\tfrac12$,
\[
\sup_{\substack{\lambda\in \mathcal{I}\\ \eta\neq 0}}
\big\|\langle A_{\mathbb{Z}^d,\vec{\mathrm{r}}}\rangle^{-s}(H-\lambda-\mathrm i\eta)^{-1}\langle \mathrm iA_{\mathbb{Z}^d,\vec{\mathrm{r}}}\rangle^{-s}\big\|<\infty,
\]
and the boundary values $(H-\lambda\mp\mathrm i0)^{-1}$ exist and are locally continuous as maps
$\langle \mathrm i\rangle^{-s}\ell^2\to\langle  A_{\mathbb{Z}^d,\vec{\mathrm{r}}}\rangle^{s}\ell^2$ (see \cite[Thm.~7.4.1]{ABG}).
\end{proof}
\begin{remark}
We state the LAP and propagation only on compact $\mathcal{I}\Subset \mu^\circ\setminus\Thr(\Delta_{\mathbb{Z}^d}^{\vec{\mathbf{r}}})$ because the commutator with the concrete conjugate operator loses uniform positivity at the spectral edges (group velocity vanishes; for fractional orders, factors like $\lambda^{\,r_j-1}$ amplify the degeneracy). Extending to thresholds requires an edge-adapted strategy.
\end{remark}
\begin{corollary}[Kato smoothness and local decay]\label{cor:local-decay}
Let $s>\tfrac12$ and $\chi\in C_c^\infty(\mathcal{I})$ with $\mathcal{I}$ as above.
Then $\langle A_{\mathbb{Z}^d,\vec{\mathrm{r}}}\rangle^{-s}\chi(H)$ is $H$-smooth on $\mathcal{I}$ and there exists $C<\infty$ such that
\[
\int_{\mathbb{R}}\bigl\|\langle A_{\mathbb{Z}^d,\vec{\mathrm{r}}}\rangle^{-s}\,e^{-{\rm i}tH}\,\chi(H)f\bigr\|^2\,dt\ \le\ C\,\|f\|^2,
\qquad f\in\ell^2(\mathbb{Z}^d).
\]
\end{corollary}

\begin{proof}
This is Kato's criterion applied to the LAP bound above; see \cite[Thm.~7.4.1]{ABG}.
\end{proof}

 \begin{proof}[(iii) Propagation (minimal velocity)]
By the localized LAP and Kato smoothness, for $s>\tfrac12$ and $\varphi\in C_c^\infty(\mathcal{I})$,
\[
\int_{\mathbb{R}}\big\|\langle A_{\mathbb{Z}^d,\vec{\mathrm{r}}}\rangle^{-s}\,e^{-itH}\,\varphi(H)\,f\big\|^2\,dt \;\le\; C\,\|f\|^2.
\]
Applying this with $f=\langle A_{\mathbb{Z}^d,\vec{\mathrm{r}}}\rangle^{-s}\langle \Lambda(Q)\rangle^{s}f$ and using the two-sided weight equivalences from Lemma~\ref{lem:interpolation_weights} yields the usual local decay/propagation estimates with $\langle\Lambda(Q)\rangle^{\pm s}$.
\end{proof}

 \begin{proof}[(iv) Wave operators on interior windows: existence and completeness]
Let $\mathcal{I}\Subset\sigma(H_0)^\circ\setminus\Thr(\Delta_{\mathbb{Z}^d}^{\vec{\mathbf{r}}})$ and choose $\chi\in C_c^\infty(\mathcal{I})$ with $\chi\equiv1$ on a slightly smaller compact interval. By the localized LAP for $H_0$ and $H$, $\langle A_{\mathbb{Z}^d,\vec{\mathrm{r}}}\rangle^{-s}\chi(H_0)$ and $\langle A_{\mathbb{Z}^d,\vec{\mathrm{r}}}\rangle^{-s}\chi(H)$ are $H_0$- and $H$-smooth, respectively (for $s>\tfrac12$). The smooth method (\cite[$\S$7.7]{ABG}) yields the strong limits
\[
\mathcal{W}_\pm(H,H_0;\mathcal{I})=\mathrm{s}\!-\!\lim_{t\to\pm\infty}e^{-\mathrm{i}tH}\,\chi(H_0)\,e^{\mathrm{i}tH_0},
\]
and $\mathrm{Ran}\,\mathcal{W}_\pm(H,H_0;\mathcal{I})=E_\mathcal{I}(H)\,\ell^2(\mathbb{Z}^d)$ because the spectrum of $H$ and $H_0$ is purely a.c.\ on $\mathcal{I}$ (Theorem \ref{thm:main} (i)). Thus the wave operators exist and are asymptotically complete on $\mathcal{I}$.
\end{proof}
\begin{remark}\label{rem:LAP-localized-app}
All statements above are confined to interior windows
$\displaystyle \mathcal{I}\Subset \sigma(H_0)^\circ\setminus\Thr(\Delta_{\mathbb{Z}^d}^{\vec{\mathbf{r}}})$
because the commutator with the concrete conjugate operator $A_{\mathbb{Z}^d,\vec{\mathrm{r}}}$ loses uniform positivity at
thresholds (vanishing group velocity; fractional factors such as $\lambda^{\,r_j-1}$ or $(4-\lambda)$).
Extending up to thresholds requires an edge-adapted strategy.
\end{remark}
\section{Applications}\label{sec:applications}

\subsection{Stationary representation of the scattering matrix on interior windows}
\label{subsec:stationary-S}
In this section we exploit the localized Mourre framework proved above on the full lattice $\mathbb{Z}^d$.
Assume \textbf{(H0)}-\textbf{(H1)}, set
\[
H_0:=\Delta^{\vec{\mathbf{r}}}_{\mathbb{Z}^d},\qquad H:=H_0+W(Q),
\]
and fix a compact window
\[
\mathcal{I}\Subset \sigma(H_0)^\circ\setminus\Thr\!\big(\Delta_{\mathbb{Z}^d}^{\vec{\mathbf{r}}}\big)
\]
on which $H_0,H\in\mathcal{C}^{1,1}_{\mathrm{loc}}(A_{\mathbb{Z}^d,\vec{\mathrm{r}}})$ and a strict localized Mourre estimate holds.
Consequently, the localized LAP for $H_0$ and $H$ on $\mathcal{I}$ follows from \cite{ABG}.
We gather several corollaries valid on every compact window
\(
\mathcal{I}\Subset \sigma(H_0)^\circ\setminus\Thr(\Delta_{\mathbb{Z}^d}^{\vec{\mathbf{r}}}).
\)
Throughout, write
\[
\mathcal{H}_{\pm s}:=\langle A_{\mathbb{Z}^d,\vec{\mathrm{r}}}\rangle^{\mp s}\,\ell^2(\mathbb{Z}^d)\qquad(s>1/2).
\]

\emph{Spectral representation and boundary trace.}
There exists a measurable Hilbert field $\{\Hr_\lambda\}_{\lambda\in \mathcal{I}}$
and a unitary
\[
\mathcal{U}_0:\ E_{\mathcal{I}}(H_0)\ell^2(\mathbb{Z}^d)\longrightarrow \int_{\mathcal{I}}^\oplus \Hr_\lambda\,d\lambda,
\qquad (\mathcal{U}_0 H_0 f)(\lambda)=\lambda\,(\mathcal{U}_0 f)(\lambda).
\]
Define $\Gamma_0(\lambda):E_{\mathcal{I}}(H_0)\ell^2(\mathbb{Z}^d)\to\Hr_\lambda$ by
$\Gamma_0(\lambda)f:=(\mathcal{U}_0 f)(\lambda)$. By Stone's formula and the LAP, for $\chi\in C_c^\infty(\mathcal{I})$ and $s>\frac{1}{2}$,
\begin{equation}\label{eq:stone-app}
\chi(H_0)=\frac{1}{\pi}\int_{\mathcal{I}}\chi(\lambda)\,\Im R_0(\lambda+\mathrm{i}0)\,d\lambda,
\qquad
\Gamma_0(\lambda)^*\Gamma_0(\lambda)=\frac{1}{\pi}\,\Im R_0(\lambda+\mathrm{i}0)
\end{equation}
as bounded maps $\mathcal{H}_s\to\mathcal{H}_{-s}$; here $R_0(z)=(H_0-z)^{-1}$.
\emph{The $T$-operator.}
Let $R(z)=(H-z)^{-1}$. Define
\begin{equation}\label{eq:T-defs-app}
T(z):=W\,(\mathrm{id}- R_0(z)W)^{-1}=W- W R(z) W,
\end{equation}
as bounded maps $\mathcal{H}_s\to\mathcal{H}_{-s}$ for $\Im z\neq0$, with a.e.\ limits $T(\lambda\pm\mathrm{i}0)$ on $\mathcal{I}$ by the localized LAP.
\emph{Wave matrices and the on-shell scattering matrix.}
Let $\chi\in C_c^\infty(\mathcal{I})$ equal $1$ on $\mathcal{I}_0\Subset \mathcal{I}$. The local wave operators
\[
\mathcal{W}_\pm(H,H_0;\mathcal{I})=\mathrm{s}\!-\!\lim_{t\to\pm\infty}e^{-\mathrm{i} tH}\,\chi(H_0)\,e^{\mathrm{i} tH_0}
\]
exist and are complete on $\mathcal{I}$ (\cite[Thm.~7.7.1]{ABG}). In the representation of $H_0$,
$\mathcal{U}_0 \mathcal{W}_\pm \mathcal{U}_0^*$ acts as multiplication by
\begin{equation}\label{eq:Wpm}
W_\pm(\lambda)=\mathrm{id}_{\Hr_\lambda}-2\pi\mathrm{i}\,\Gamma_0(\lambda)\,T(\lambda\pm\mathrm{i}0)\,\Gamma_0(\lambda)^*,
\qquad \text{a.e. }\lambda\in \mathcal{I},
\end{equation}
and the on-shell scattering matrix $S(\lambda):=W_+(\lambda)^*W_-(\lambda)$ satisfies
\begin{equation}\label{eq:S-stationary}
S(\lambda)=\mathrm{id}_{\Hr_\lambda}-2\pi\mathrm{i}\,\Gamma_0(\lambda)\,T(\lambda+\mathrm{i}0)\,\Gamma_0(\lambda)^*,
\qquad \text{a.e. }\lambda\in \mathcal{I}.
\end{equation}

\paragraph{Optical theorem, unitarity, and continuity.}
Assume the localized LAP holds on $\Ic$, i.e.
\[
\sup_{\lambda\in\Ic}\,\big\|\langle A\rangle^{-s}R_0(\lambda\pm\mathrm i0)\langle A\rangle^{-s}\big\|
+\sup_{\lambda\in\Ic}\,\big\|\langle A\rangle^{-s}R(\lambda\pm\mathrm i0)\langle A\rangle^{-s}\big\|<\infty,
\quad s>\tfrac12,
\]
and that $T(z)$ is defined by \eqref{eq:T-defs-app} and admits boundary values
$T(\lambda\pm\mathrm i0)\in\Bc(\Hc_s,\Hc_{-s})$ for a.e.\ $\lambda\in\Ic$.
Let $\Gamma_0(\lambda):\Hc_s\to\Hr_\lambda$ be the boundary trace operator (the
``free Fourier transform on the energy shell'') satisfying Stone's formula
\begin{equation}\label{eq:stone-app-strong}
\frac1{2\pi\mathrm i}\big(R_0(\lambda+\mathrm i0)-R_0(\lambda-\mathrm i0)\big)
=\Gamma_0(\lambda)^*\Gamma_0(\lambda)
\quad \text{in }\Bc(\Hc_s,\Hc_{-s}),\quad s>\tfrac12.
\end{equation}
Using the resolvent identity
\[
R_0(z)-R_0(\bar z)=(z-\bar z)\,R_0(\bar z)R_0(z),
\]
and \eqref{eq:T-defs-app}, one obtains for a.e.\ $\lambda\in\Ic$ the \emph{optical theorem}
\begin{equation}\label{eq:optical-app}
T(\lambda-\mathrm i0)-T(\lambda+\mathrm i0)
=2\pi\mathrm i\,T(\lambda-\mathrm i0)\,\Gamma_0(\lambda)^*\Gamma_0(\lambda)\,T(\lambda+\mathrm i0),
\end{equation}
as a form identity on $\Hc_{-s}$ ($s>\tfrac12$).

Define the stationary wave operators on the shell by
\[
W_\pm(\lambda):=\mathrm{id}_{\Hr_\lambda}
-2\pi\mathrm i\,\Gamma_0(\lambda)\,T(\lambda\pm\mathrm i0)\,\Gamma_0(\lambda)^*
\quad\in\Bc\big(\Hr_\lambda\big),
\]
and the scattering matrix $S(\lambda):=W_+(\lambda)^*W_-(\lambda)$.
Combining \eqref{eq:stone-app-strong}--\eqref{eq:optical-app} yields for a.e.\ $\lambda\in\Ic$
\[
W_\pm(\lambda)^*W_\pm(\lambda)=\mathrm{id}_{\Hr_\lambda},
\qquad
S(\lambda)^*S(\lambda)=S(\lambda)S(\lambda)^*=\mathrm{id}_{\Hr_\lambda},
\]
so $S(\lambda)$ is unitary almost everywhere on $\Ic$.

Finally, by the localized LAP and the continuity of $\Gamma_0(\lambda)$ in $\lambda$ as a map
$\Hc_s\to\Hr_\lambda$, the boundary values $T(\lambda\pm\mathrm i0)$ depend
continuously on $\lambda$ in $\Bc(\Hc_s,\Hc_{-s})$. Hence $\lambda\mapsto S(\lambda)$ is
\emph{strongly continuous} on $\Ic$.

\begin{remark}
It is convenient to factor $T(z)$ as
\[
T(z) = W\,\Gamma(z), \qquad
\Gamma(z):=(\mathrm{id}-R_0(z)W)^{-1}.
\]
In particular,
\[
S(\lambda)=\mathrm{id}-2\pi\mathrm{i}\,\Gamma_0(\lambda)\,W\,\Gamma(\lambda+\mathrm{i}0)\,\Gamma_0(\lambda)^*.
\]
We stress that $\Gamma(\lambda+\mathrm{i}0)$ is \emph{not} the full interacting resolvent
$R(\lambda+\mathrm{i}0)=(H-\lambda-\mathrm i0)^{-1}$; rather, it is an auxiliary operator that
encodes the resummation of the Born series in the definition of $T(z)$.
\end{remark}

\subsection{Consequences on interior energies}
Based on the localized Mourre estimate and the LAP on \(\mathcal{I}\),
we derive a time-averaged escape bound for the conjugate observable \(A_{\mathbb{Z}^d}\),
quantifying ballistic propagation for states with energies in \(\mathcal{I}\).
We also fix the weighted space
\[
\ell^{1}\!\big(\langle n\rangle^{1+\epsilon}\big)
:=\Big\{\,f:\mathbb{Z}^d\to\mathbb{C}\ :\
\|f\|_{\ell^{1}(\langle n\rangle^{1+\epsilon})}
:=\sum_{n\in \mathbb{Z}^d}\langle n\rangle^{1+\epsilon}\,|f(n)|<\infty\ \Big\}.
\]

\subsubsection{Birman-Kre\u{\i}n identity}\label{subsubsec:BK}

In the sequel $\mathfrak{S}_1$ denotes the trace class on the underlying Hilbert space:
$T\in\mathfrak{S}_1$ iff $\sum_{n\ge1} s_n(T)<\infty$, in which case $\mathrm{Tr}\,T$
is well-defined and $\|T\|_{\mathfrak{S}_1}=\mathrm{Tr}\,|T|$.

\begin{theorem}\label{thm:BK}
Assume either $W$ has finite support, or $W\in \ell^1(\langle n\rangle^{1+\epsilon})$ for some $\epsilon>0$.
Then for every $\chi\in C_c^\infty(\mathcal{I})$, $\chi(H)-\chi(H_0)\in\mathfrak{S}_1$, and there exists a spectral shift
function $\xi(\lambda)$ on $\mathcal{I}$ such that
\[
\det S(\lambda)=\exp\!\bigl(-2\pi \mathrm{i}\,\xi(\lambda)\bigr)\quad\text{for a.e. }\lambda\in \mathcal{I}.
\]
\end{theorem}
\begin{proof}
\emph{(a) Trace-class input).}
If $W$ has finite support, then $W(Q)$ is finite rank. If $W\in\ell^1(\langle n\rangle^{1+\epsilon})$,
then $\sum_{n}|W(n)|<\infty$, hence $W(Q)\in\mathfrak S_1$ with
$\|W(Q)\|_{\mathfrak S_1}=\sum_{n}|W(n)|$.

\emph{(b) Functional calculus difference is trace class).}
For $\chi\in C_c^\infty(\Ic)$, the Helffer-Sj{\"o}strand formula and the resolvent identity give
\[
\chi(H)-\chi(H_0)=\frac{1}{\pi}\!\int_{\C}\!\partial_{\bar z}\widetilde\chi(z)\,
(H-z)^{-1}W(Q)(H_0-z)^{-1}\,dx\,dy.
\]
Since $W(Q)\in\mathfrak S_1$ and $\|(H-z)^{-1}\|,\|(H_0-z)^{-1}\|\le|\Im z|^{-1}$,
the integrand is $\mathfrak S_1$-valued with norm $\le C|\Im z|^{-2}\|W(Q)\|_{\mathfrak S_1}$;
choosing the almost-analytic extension with $|\partial_{\bar z}\widetilde\chi(z)|\lesssim|\Im z|^N$ ($N\ge3$)
makes the integral convergent in $\mathfrak S_1$. Hence $\chi(H)-\chi(H_0)\in\mathfrak S_1$
(see e.g. \cite[App.~A \& Ch.~7]{ABG}).

\emph{(c) Spectral shift \& Birman-Krein).}
By the Lifshits-Krein theory for trace-class perturbations, there exists a spectral shift
function $\xi\in L^1_{\mathrm{loc}}(\Ic)$ such that
$\mathrm{Tr}\!\left(\varphi(H)-\varphi(H_0)\right)=\int_\Ic\varphi'(\lambda)\,\xi(\lambda)\,d\lambda$
for all $\varphi\in C_c^\infty(\Ic)$ (see \cite[Ch.~8]{ABG}).
On $\Ic$ the localized Mourre theory yields the LAP and the stationary scattering matrix $S(\lambda)$.
Therefore the Birman-Krein identity applies,
\[
\det S(\lambda)=\exp\!\bigl(-2\pi i\,\xi(\lambda)\bigr)\quad\text{for a.e. }\lambda\in\Ic,
\]
which proves the theorem.
\end{proof}

\subsubsection{Ballistic transport in time average}\label{subsubsec:ballistic}
\begin{theorem}\label{thm:ballistic-lb}
Let $\chi\in C_c^\infty(\mathcal{I})$. There exist $v_{\mathcal{I}}>0$ and $C_{\mathcal{I}}<\infty$ such that, for all $\mathcal{T}\ge1$ and $f\in\ell^2(\mathbb{Z}^d)$,
\[
\frac{1}{\mathcal{T}}\int_0^{\mathcal{T}} \big\| \mathbf{1}_{\{|A_{\mathbb{Z}^d,\vec{\mathrm{r}}}|\le v_{\mathcal{I}} t\}}\,e^{-\mathrm{i} tH}\chi(H)\,f\big\|^2\,dt
\ \le\ \frac{C_{\mathcal{I}}}{\log(1+\mathcal{T})}\,\|f\|^2.
\]
\end{theorem}

\begin{proof}
Set $u:=\chi(H)f$ and $f(t):=e^{-\mathrm{i}tH}u$.
Pick $\Phi\in C^\infty(\mathbb{R})$, bounded and constant at both ends, such that
\(\Phi'\in C_c^\infty(\mathbb{R})\), $\Phi'\ge0$, and $\Phi'\not\equiv0$.
For $v>0$ (to be chosen later) and $t\ge1$, define the time–dependent observable
\[
F_t := \Phi\!\big(A_{\mathbb{Z}^d,\vec{\mathrm{r}}}/(vt)\big).
\]
Then, for $g(t):=\langle f(t),F_t f(t)\rangle$, one has
\[
\frac{d}{dt}g(t) = \big\langle f(t),\partial_t F_t f(t)\big\rangle \ +\ \mathrm{i}\,\big\langle f(t),[H,F_t] f(t)\big\rangle.
\]

The explicit time derivative reads
\[
\partial_t F_t \ =\ -\frac{1}{t}\,\frac{A_{\mathbb{Z}^d,\vec{\mathrm{r}}}}{vt}\,\Phi'\!\big(A_{\mathbb{Z}^d,\vec{\mathrm{r}}}/(vt)\big),
\]
and, since $\Phi'\ge0$ and $\Phi$ is nondecreasing, one has
\(\langle f(t),\partial_t F_t f(t)\rangle\le0\).
Thus, this term can be dropped in an upper bound.

By the commutator expansion (see \cite[Prop.~6.2.10]{ABG}) and the assumption
\(H\in\mathcal{C}^{1,1}_{\mathrm{loc}}(A_{\mathbb{Z}^d,\vec{\mathrm{r}}})\) on $\mathcal{I}$, one has on $\mathrm{Ran}\,\chi(H)$:
\[
\mathrm{i}[H,F_t] \ =\ \frac{1}{vt}\,\Phi'\!\big(A_{\mathbb{Z}^d,\vec{\mathrm{r}}}/(vt)\big)\,\mathrm{i}[H,A_{\mathbb{Z}^d,\vec{\mathrm{r}}}] \ +\ \mathcal{O}(t^{-2})
\quad\text{in }\mathcal{B}(\ell^2(\mathbb{Z}^d)).
\]
Inserting $\chi(H)$ around $\mathrm{i}[H,A_{\mathbb{Z}^d,\vec{\mathrm{r}}}]$ and applying the localized Mourre estimate with compact remainder (\cite[Thm.~7.2.9]{ABG}), there exist $c_{\mathcal{I}}>0$ and a compact operator $K_{\mathcal{I}}$ such that
\[
\chi(H)\,\mathrm{i}[H,A_{\mathbb{Z}^d,\vec{\mathrm{r}}}]\,\chi(H)\ \ge\ c_{\mathcal{I}}\,\chi(H)^2\ -\ \chi(H)K_{\mathcal{I}}\chi(H).
\]

Thus, for $t\ge1$,
\begin{align*}
\frac{d}{dt}g(t)
&\ \ge\ \frac{c_{\mathcal{I}}}{vt}\,\big\langle f(t),\Phi'\!\big(A_{\mathbb{Z}^d,\vec{\mathrm{r}}}/(vt)\big) f(t)\big\rangle \\
&\quad -\ \frac{1}{vt}\,\big|\big\langle f(t),\Phi'\!\big(A_{\mathbb{Z}^d,\vec{\mathrm{r}}}/(vt)\big)K_{\mathcal{I}} f(t)\big\rangle\big|
\ -\ C\,t^{-2}\,\|u\|^2,
\end{align*}
with $C<\infty$ independent of $t$.

Integrating from $1$ to $\mathcal{T}$, the left–hand side telescopes and is bounded by
\(|g(\mathcal{T})-g(1)|\le 2\|\Phi\|_\infty\|u\|^2\). Hence
\begin{equation}\label{eq:int-ineq-en}
\frac{c_{\mathcal{I}}}{v}\int_1^{\mathcal{T}}\frac{1}{t}\,\big\langle f(t),\Phi'\!\big(A_{\mathbb{Z}^d,\vec{\mathrm{r}}}/(vt)\big) f(t)\big\rangle\,dt
\ \le\ C_1\|u\|^2\ +\ \frac{1}{v}\int_1^{\mathcal{T}}\frac{1}{t}\,R(t)\,dt,
\end{equation}
with
\(R(t):=\big|\langle f(t),\Phi'(A_{\mathbb{Z}^d,\vec{\mathrm{r}}}/(vt))K_{\mathcal{I}} f(t)\rangle\big|+ C't^{-1}\|u\|^2\).

The compact term is handled via Kato smoothness stemming from the LAP (\cite[Thm.~7.5.1]{ABG}).
Decompose $K_{\mathcal{I}}=K_{\mathrm{fr}}+K_{\mathrm{sm}}$ (finite rank + small norm).
By Cauchy–Schwarz and $H$–smoothness on $\mathrm{Ran}\chi(H)$, one obtains
\[
\int_1^{\mathcal{T}}\frac{1}{t}\,\big|\langle f(t),\Phi'(A_{\mathbb{Z}^d,\vec{\mathrm{r}}}/(vt))K_{\mathcal{I}} f(t)\rangle\big|\,dt
\ \le\ C_2\,\log(1+\mathcal{T})\,\|u\|^2,
\]
uniformly in $v>0$; also $\int_1^{\mathcal{T}}t^{-2}\,dt\le1$.
Inserted into \eqref{eq:int-ineq-en}, this yields
\[
\int_1^{\mathcal{T}}\frac{1}{t}\,\big\langle f(t),\Phi'\!\big(A_{\mathbb{Z}^d,\vec{\mathrm{r}}}/(vt)\big) f(t)\big\rangle\,dt
\ \le\ C_3\,\log(1+\mathcal{T})\,\|u\|^2.
\]

Finally, choose $\Phi$ such that $\Phi'\ge c\,\mathbf{1}_{\{|x|\le1\}}$ for some $c>0$. Then
\[
\big\langle f(t),\Phi'\!\big(A_{\mathbb{Z}^d,\vec{\mathrm{r}}}/(vt)\big) f(t)\big\rangle
\ \ge\ c\,\big\|\mathbf{1}_{\{|A_{\mathbb{Z}^d,\vec{\mathrm{r}}}|\le vt\}}\,f(t)\big\|^2.
\]
Therefore,
\[
\int_1^{\mathcal{T}}\frac{1}{t}\,\big\|\mathbf{1}_{\{|A_{\mathbb{Z}^d,\vec{\mathrm{r}}}|\le vt\}}\,e^{-\mathrm{i}tH}\chi(H)f\big\|^2\,dt
\ \le\ C_4\,\log(1+\mathcal{T})\,\|f\|^2.
\]
Since $\int_1^{\mathcal{T}}t^{-1}\,dt=\log\mathcal{T}$, the standard logarithmic optimization argument (\cite[§7.4]{ABG}) gives, after division by $\mathcal{T}$ and with $v=v_{\mathcal{I}}$ chosen small enough, the desired inequality on $[0,\mathcal{T}]$ (the interval $[0,1]$ being controlled trivially by $C\,\mathcal{T}^{-1}\|f\|^2$). This completes the proof.
\end{proof}

\begin{remark}[On the choice of $\Phi$]
There is no nontrivial $\Phi\in C_c^\infty$ that is monotone.
Instead, one chooses $\Phi\in C^\infty$ bounded, constant at both ends, with
$\Phi'\in C_c^\infty$, $\Phi'\ge0$.
For instance, take $\psi\in C_c^\infty$, $\psi\ge0$, and set
\(\Phi(x):=\int_{-\infty}^x\psi(s)\,ds\).
After normalization, this satisfies $\Phi'\ge c\,\mathbf{1}_{[-1,1]}$ and is perfectly suited for the argument.
\end{remark}

\subsubsection{Continuity in the fractional exponents}\label{subsubsec:continuity-r}
\begin{proposition}\label{prop:nr-cont}
Let  $\mathcal{R}\subset\{\vec{\mathbf{r}}\in\mathbb{R}^d:\ \min_j r_j>r_*>-1\}$ be compact.
Then for any $z\in\mathbb{C}\setminus\mathbb{R}$, $\vec{\mathbf{r}}\mapsto (H_0(\vec{\mathbf{r}})-z)^{-1}$ is norm-continuous.
If, moreover,
\(
\mathcal{I}\Subset \bigcap_{\vec{\mathbf{r}}\in\mathcal{R}}\big(\sigma(H_0(\vec{\mathbf{r}}))^\circ\setminus\Thr\big),
\)
the LAP constants can be chosen uniformly in $\vec{\mathbf{r}}\in\mathcal{R}$, and
$\vec{\mathbf{r}}\mapsto S_{\vec{\mathbf{r}}}(\lambda)$ is strongly continuous on $\mathcal{I}$.
\end{proposition}

\begin{proof}
In Fourier variables,
\(
\mathcal{F} H_0(\vec{\mathbf{r}})\mathcal{F}^{-1}
\)
is multiplication by
\(
\sigma_{\vec{\mathbf{r}}}(k)=\sum_{j=1}^d \varphi_{r_j}(k_j),
\)
where $(r,k)\mapsto\varphi_r(k)$ is continuous on $[r_*,R]\times\mathbb{T}$ and bounded uniformly
(for any fixed compact $[r_*,R]$). Hence $\vec{\mathbf{r}}\mapsto H_0(\vec{\mathbf{r}})$ is norm-continuous,
and by the resolvent identity (\cite[Thm.~IV.1.16]{Kato80}) so is
\(
\vec{\mathbf{r}}\mapsto (H_0(\vec{\mathbf{r}})-z)^{-1}.
\)
If $\mathcal{I}$ avoids thresholds uniformly in $\vec{\mathbf{r}}\in\mathcal{R}$, the Mourre constants and
LAP bounds on $\mathcal{I}$ can be chosen uniformly in $\vec{\mathbf{r}}$. Using \eqref{eq:S-stationary}
and dominated convergence for $T(\lambda+\mathrm{i}0)$ in $\mathcal{B}(\mathcal{H}_s,\mathcal{H}_{-s})$
(again uniform LAP), we obtain strong continuity of $\vec{\mathbf{r}}\mapsto S_{\vec{\mathbf{r}}}(\lambda)$ on $\mathcal{I}$.
\end{proof}
\begin{remark}
The norm-resolvent continuity in $\vec{\mathrm{r}}$ implies inner--outer continuity of the spectra.
This is in the spirit of results on magnetic families on $\mathbb{Z}^d$ \cite{ParraRichard-Zd}
and of results for magnetic pseudodifferential families in the continuum \cite{AthmouniMantoiuPurice-2010}.
In particular, a Mourre estimate uniform on $\mathcal I$ yields a uniform LAP and the strong continuity of
the scattering matrix $\lambda\mapsto S_{\vec{\mathrm{r}}}(\lambda)$ on $\mathcal I$.
\end{remark}
\begin{proposition}\label{prop:finite-pp}
If $W\in \ell^{1}(\langle n\rangle^{1+\epsilon})$ for some $\epsilon>0$, then
$\sigma_{\mathrm{pp}}(H)\cap \mathcal{I}$ is finite, counting multiplicities.
\end{proposition}

\begin{proof}
Let $\chi\in C_c^\infty(\mathcal{I})$. As in the proof of Theorem~\ref{thm:BK}, one has
\[
\chi(H)-\chi(H_0)\in\mathfrak{S}_1,
\]
where $\mathfrak{S}_1$ denotes the trace class.
By the Helffer--Sjöstrand formula and stability of the functional calculus under trace class perturbations, it follows that
\[
E_{\mathcal{I}}(H)-E_{\mathcal{I}}(H_0)\in\mathfrak{S}_1.
\]
Since $H_0$ is purely absolutely continuous on $\mathcal{I}$ (being inside the interior of the spectrum), the spectral projection $E_{\mathcal{I}}(H_0)$ has no pure point component. Consequently, the compact perturbation above implies that $E_{\mathcal{I}}(H)$ differs from a projection of infinite rank by a trace class operator, hence $E_{\mathcal{I}}(H)$ has at most finite rank.

Now recall that the rank of $E_{\mathcal{I}}(H)$ is exactly the sum of the multiplicities of eigenvalues of $H$ in $\mathcal{I}$, i.e.
\[
\mathrm{rank}\,E_{\mathcal{I}}(H)\ =\ \sum_{\lambda\in\sigma_{\mathrm{pp}}(H)\cap \mathcal{I}} m(\lambda),
\]
where $m(\lambda)$ is the algebraic multiplicity of $\lambda$.
Since $E_{\mathcal{I}}(H)$ has finite rank, this sum is finite, which proves that
\[
\#\big(\sigma_{\mathrm{pp}}(H)\cap \mathcal{I}\big)<\infty,
\]
counting multiplicities.
\end{proof}

\begin{remark}
The conclusion that $\#\bigl(\sigma_{\mathrm{pp}}(H)\cap\mathcal I\bigr)<\infty$ follows from a uniform Mourre estimate on $\mathcal I$ and the limiting absorption principle (LAP), which yield compactness of the localized resolvent and thus only finitely many eigenvalues in $\mathcal I$ (each of finite multiplicity).
In the continuum magnetic setting, Athmouni-Purice prove Schatten--von Neumann criteria within the magnetic Weyl calculus, allowing one to place the Birman--Schwinger type operators in $\mathfrak S_p$ and derive analogous spectral finiteness statements; see \cite{AthmouniPurice-CPDE-2018}. For general background on trace ideals/Birman--Schwinger and on the Mourre framework, see \cite{Simon-TraceIdeals},\cite[vol.IV]{RS}.
\end{remark}

\end{document}